\documentclass[12pt, reqno]{amsart}

\usepackage{amsmath}
\usepackage{amssymb,amsfonts,amscd,amsthm}
\usepackage{latexsym}
\usepackage{mathrsfs}
\usepackage{amsaddr}
\usepackage{color}

\usepackage{hyperref}

\newcommand\al{\alpha}

\newcommand\Ga{\Gamma}
\newcommand\de{\delta}

\newcommand\om{\omega}
\newcommand\Om{\Omega}

\newcommand\la{\lambda}
\newcommand\La{{\Lambda}}
\newcommand{\eps}{\varepsilon}

\newcommand\si{\sigma}

\newcommand\R{\mathbb R}
\newcommand\C{\mathbb C}

\newcommand\Z{\mathbb Z}
\newcommand\Y{\mathbb Y}

\renewcommand\L{\mathbb  L}

\renewcommand\P{\mathscr P}
\newcommand\XX{\mathcal X}
\newcommand\YY{\mathcal Y}
\newcommand\LL{\mathcal L}

\newcommand\const{\operatorname{const}}

\newcommand\Sym{\operatorname{Sym}}

\newcommand\Sign{\operatorname{Sign}}
\newcommand\DSign{\operatorname{DSign}}

\newcommand\Conf{\operatorname{Conf}}

\newcommand\wt{\widetilde}

\newcommand\Int{\mathscr I}
\newcommand\LaN{\La^N_{N-1}}
\newcommand\ccdot{\,\cdot\,}

\DeclareMathOperator*\Res{Res}
\newcommand\pprec{\prec\!\!\prec}

\newtheorem{theorem}{Theorem}[section]
\newtheorem{proposition}[theorem]{Proposition}
\newtheorem{lemma}[theorem]{Lemma}
\newtheorem{corollary}[theorem]{Corollary}
\newtheorem*{theoremA}{Theorem A}
\newtheorem*{theoremB}{Theorem B}
\newtheorem*{theoremC}{Theorem C}

\theoremstyle{definition}
\newtheorem{definition}[theorem]{Definition}
\newtheorem{remark}[theorem]{Remark}
\newtheorem{example}[theorem]{Example}

\numberwithin{equation}{section}

\begin{document}

\title[]{Macdonald polynomials and extended Gelfand--Tsetlin graph}

\author{Grigori Olshanski${}^{1,2,3}$}
\address{${}^1$Institute for Information Transmission Problems of the Russian Academy of Sciences, Moscow, Russia.\\ ${}^2$Skolkovo Institute of Science and Technology, Moscow, Russia.\\
${}^3$ National Research University Higher School of Economics, Moscow, Russia.\\{\rm email: olsh2007@gmail.com}}

\begin{abstract}

Using Okounkov's $q$-integral representation of Macdonald polynomials we construct an infinite sequence $\Om_1,\Om_2,\Om_3,\dots$ of countable sets linked by transition probabilities from $\Om_N$ to $\Om_{N-1}$ for each $N=2,3,\dots$. The elements of the sets $\Om_N$ are the vertices of the extended Gelfand--Tsetlin graph, and the transition probabilities depend on the two Macdonald parameters, $q$ and $t$. These data determine a family of Markov chains, and the main result is the description of their entrance boundaries. This work has its origin in asymptotic representation theory. In the subsequent paper, the main result is applied to large-$N$ limit transition in $(q,t)$-deformed $N$-particle beta-ensembles. 
\end{abstract}


\date{}

\maketitle

\tableofcontents

\section{Introduction}

One of the basic problems of representation theory is the study of irreducible spherical unitary representations and their spherical functions. It is well known that spherical functions on classical symmetric spaces of compact type are expressed through multivariate orthogonal polynomials --- the Jack and Jacobi symmetric polynomials with certain special values of parameters. 

The notion of spherical function also makes sense for infinite-dimensional symmetric spaces of the form
$$
G_\infty/K_\infty =\varinjlim G_N/K_N,
$$
where
$$
\cdots \to G_N/K_N \to G_{N+1}/K_{N+1}\to\cdots\,.
$$
is a chain of nested finite-dimensional  symmetric spaces. 

There are 10 infinite series $\{G_N/K_N\}$ of compact classical symmetric spaces of growing rank, with natural embeddings $G_N/K_N \to G_{N+1}/K_{N+1}$. For each such series, there are plenty of indecomposable spherical functions on $G_\infty/K_\infty$ indexed by  countably many continuous parameters. It turns out that the description of spherical functions on $G_\infty/K_\infty$ is equivalent to finding the entrance boundary of a Markov chain obtained as a \emph{dualization} of the chain 
$\{G_N/K_N\}$.

This reformulation is important for (at least) two reasons: 

\smallskip

(1) the initial problem setting can be extrapolated to the case of Jack/Jacobi symmetric polynomials with \emph{general} parameters;

(2)   the entrance boundary of the resulting Markov chain can be found by tools of algebraic combinatorics   (see Okounkov and the author \cite{OO-IMRN}, \cite{OO-2006}).

\smallskip
 
The goal of the present paper is to move up the theory to the level of Macdonald polynomials. The main results are three theorems denoted as Theorem A, Theorem B, and Theorem C. 

\medskip

$\bullet$ In Theorem A (see section \ref{results1}) we construct a certain Markov chain depending on $q$ and $t$; we call it the \emph{extended Macdonald chain}. The theorem is deduced from Okounkov's  $q$-integral representation of Macdonald polynomials \cite{Ok-CM}.

\smallskip

$\bullet$ Theorem B (see section \ref{results2}) describes the entrance boundary of the extended Macdonald  chain. 

\smallskip

$\bullet$ Theorem C (see section \ref{results3}) is an approximation theorem. It shows that each probability measure on the boundary of the extended Macdonald chain can be obtained, in a canonical way, via a large-$N$ limit transition from random $N$-particle systems.

\medskip

These results are applied in \cite{Ols-MacdonaldTwo} to constructing a $(q,t)$-deformed combinatorial version of harmonic analysis on $U(\infty)$.   

We proceed to a more detailed description of the contents of the paper. 

\subsection{Formalism of projective chains}\label{sect1.1}
Recall that a (rectangular) matrix is said to be \emph{stochastic} if its entries are nonnegative and the row sums are equal to $1$. Stochastic matrices are a particular case of \emph{Markov kernels} (Meyer \cite{Meyer}). A Markov kernel is a map $L: S\to \P(S')$, where $S$ and $S'$ are two Borel (=measurable) spaces and $\P(S')$ denotes the space of probability measures on $S'$. When both $S$ and $S'$ are finite or countable sets, $L$ is given by a stochastic matrix of format $S\times S'$. That is, its rows are indexed by $S$ and the columns are indexed by $S'$. 

Informally, one can treat $L$ as  a `generalized map' from $S$ to $S'$. We denote such a surrogate of map by a dash arrow, $S\dasharrow S'$. 

By a \emph{projective chain} we mean an infinite sequence $S_1, S_2, S_3, \dots$ of finite or countable sets linked by stochastic matrices $L^2_1, L^3_2, \dots$,  where the matrix $L^N_{N-1}$ has format $S_N\times S_{N-1}$. This is symbolically represented by the diagram
\begin{equation}\label{eq1.A}
S_1\stackrel{L^2_1}{\dashleftarrow}S_2\stackrel{L^3_2}{\dashleftarrow}S_3\stackrel{L^4_3}{\dashleftarrow}\dots\,.
\end{equation}

For any projective chain one can define, in a canonical way, its \emph{boundary}. It is a Borel space $S_\infty$, which is linked to the sets $S_N$ via Markov kernels $L^\infty_N$ satisfying the relations $L^\infty_N L^N_{N-1}=\La^\infty_{N-1}$ (a composition of Markov kernels is read from left to right). The precise definition is given in section \ref{sect7.1}.  

One can regard the diagram \eqref{eq1.A} as a non-stationary Markov chain with discrete time $N$ ranging in reverse direction (from $+\infty$ to $1$), time-dependent state spaces $S_N$, and transition kernels $L^N_{N-1}$; then $S_\infty$ is identified with what may be called the \emph{entrance boundary} of that chain, or the set of \emph{extreme entrance laws} in the terminology of Dynkin \cite[\S10.1]{Dynkin}. A comprehensive discussion is contained in Winkler's monograph \cite[ch. 4]{Winkler}. 

The boundary $S_\infty$ can also be interpreted as the inverse limit of \eqref{eq1.A} in the category-theoretical sense. The corresponding category is formed by standard Borel spaces (as objects) and Markov kernels (as morphisms), see \cite[ch. 4]{Winkler}. 

Here is an illustrative example. 

\begin{example}[Boundary of Pascal graph]
Let $S_N:=\{0,1,\dots, N\}$ and $L^N_{N-1}$ be the two-diagonal matrix with the entries 
$$
L^N_{N-1}(n,n-1)=\frac nN, \qquad L^N_{N-1}(n,n):=\frac{N-n}N, 
$$
and all other entries being equal to $0$. One can show that the boundary $S_\infty$ of this chain is the closed interval $[0,1]$ with the  Markov kernels from $S_\infty\dasharrow S_N$, $N=1,2,\dots,$  given by 
$$
L^\infty_N(x,n)=\binom Nn x^n(1-x)^{N-n}, \qquad n=0,1,\dots,N.
$$
This result is in fact equivalent to classical de Finetti theorem (\cite[section 5]{BO-2017}). It is also related to other classical topics --- Bernstein polynomials and Hausdorff moment problem (Feller \cite[ch. VII]{Feller}).
\end{example}

Even in this simple example, finding the boundary requires some work. For more sophisticated projective chains, this task may require considerable efforts.

\subsection{Projective chains in representation theory}

In the body of the text there are no group representations, we work exclusively with symmetric polynomials and symmetric functions. However, because the extended Macdonald chain originated as a generalization of some representation-theoretic constructions, it makes sense to tell a little about these constructions --- otherwise the problem setup will not be sufficiently motivated. 

Let $G\supset K$ be finite or separable compact groups forming a \emph{Gelfand pair} (Bump \cite[\S45]{Bump}); for instance, one may suppose that $G/K$ is a \emph{symmetric space of compact type} (Helgason \cite{H}). The \emph{spherical dual of $(G,K)$} is the set of indecomposable positive definite normalized functions $G\to\C$, constant on double $K$-cosets. This is a finite or countable set, which we denote by $\Om(G,K)$. It parameterizes the irreducible \emph{spherical} representations --- the irreducible unitary representations of $G$ possessing a $K$-invariant vector. 

By a \emph{morphism  $(G,K)\to(G',K')$ of Gelfand pairs} we mean a group homomorphism $\phi:G\to G'$ such that $\phi(K)\subseteq K'$. For noncommutative groups  one cannot define a natural dual map $\phi^*$ from $\Om(G',K')$ to $\Om(G,K)$. However, there is a reasonable substitute ---  a `generalized map'  $\Om(G',K')\dasharrow \Om(G,K)$ given by a stochastic matrix of format $\Om(G',K')\times\Om(G,K)$. 

This stochastic matrix, which we denote by $L^{G'}_G$,  is defined in the following way. Given a function $\om'\in\Om(G',K')$, its composition with $\phi:G\to G'$ produces a positive definite normalized function $\om'\circ\phi$ on $G$; the latter is uniquely written as a convex combination of indecomposable spherical functions of $(G,K)$ with certain coefficients $L^{G'}_G(\om',\om)$,
\begin{equation}\label{eq1.D}
\om'\circ\phi=\sum_{\om\in\Om(G,K)}L^{G'}_G(\om',\om)\om,
\end{equation} 
and these coefficients are just the entries of the matrix $L^{G'}_G$. 

Now suppose we are given an infinite sequence 
\begin{equation}\label{eq1.B}
(G_1,K_1)\to (G_2,K_2)\to (G_3,K_3)\to\dots
\end{equation}
of growing (finite or compact) Gelfand pairs. It gives rise to the dual projective chain 
\begin{equation}\label{eq1.C}
\Om(G_1,K_1)\dashleftarrow \Om(G_2,K_2)\dashleftarrow \Om(G_3,K_3)\dashleftarrow\dots
\end{equation}
which has a certain boundary. 

On the other hand, consider the inductive limit groups $G_\infty:=\varinjlim G_N$ and $K_\infty:=\varinjlim K_N$. In general, these are no longer compact groups, but the pair $(G_\infty,K_\infty)$ is still a Gelfand pair (in the sense explained in Olshanski \cite{Ols-1990}). Next,  the corresponding spherical dual $\Om(G_\infty,K_\infty)$ is defined in exactly the same way as above, and the elements of $\Om(G_\infty,K_\infty)$ parametrize the irreducible spherical representations of $(G_\infty,K_\infty)$. 

As pointed out above, there exist 10 series of classical symmetric spaces $G_N/K_N$ of compact type, and for each series, the spherical dual $\Om(G_\infty,K_\infty)$ is known (Olshanski \cite{Ols-1990}, Pickrell \cite{Pickrell}).

It is a direct consequence of definitions that  $\Om(G_\infty,K_\infty)$ can be identified with the boundary of the projective chain \eqref{eq1.C}. Thus, this boundary has a representation-theoretic meaning.

\subsection{Projective chains related to Jack polynomials}
Introduce some notation: 

\smallskip

$\bullet$ $\Sign(N)$ is the set of \emph{signatures of length $N$}; these are vectors $a=(a_1,\dots,a_N)\in\Z^N$ such that $a_1\ge\dots\ge a_N$ (the coordinates $a_i$ may be of arbitrary sign);

$\bullet$ $\tau$ is an arbitrary positive real number --- the parameter of Jack polynomials ($\tau$ is inverse to the parameter $\al$ used in Macdonald's book \cite{M});

$\bullet$ $P_a(u_1,\dots,u_N;\tau)$ is the Jack polynomial indexed by a signature $a\in\Sign(N)$ (in general, $P_a(\ccdot;\tau)$ is a Laurent polynomial; the definition given in \cite[ch. VI, \S10]{M} extends to the case of Laurent polynomials without difficulty);

$\bullet$ $u_1,\dots,u_N$ are variables;

$\bullet$ $(1^N)=(1,\dots,1)$ ($N$ times). 

\smallskip

For three special values $\tau=\frac12, 1, 2$, the normalized Jack polynomials 
$$
\frac{P_a(u_1,\dots,u_N;\tau)}{P_a((1^N); \tau)}
$$
give indecomposable spherical functions for the symmetric spaces
$$
U(N)/O(N), \qquad U(N)\times U(N)/ U(N), \qquad U(2N)/Sp(N),
$$
respectively. In terms of Jack polynomials, the expansion \eqref{eq1.D} takes the form 
\begin{multline}\label{eq1.E}
\left.\frac{P_a(u_1,\dots,u_N;\tau)}{P_a((1^N); \tau)}\right|_{u_N=1}\\
=\sum_{b\in\Sign(N-1)} L^N_{N-1}(a,b;\tau)\frac{P_b(u_1,\dots,u_{N-1};\tau)}{P_b((1^{N-1});\tau)}, \qquad a\in\Sign(N),
\end{multline}
and the coefficients $L^N_{N-1}(a,b;\tau)$ defined by \eqref{eq1.E} form a matrix $L^N_{N-1}$ of format $\Sign(N)\times\Sign(N-1)$ depending on $\tau$. 

A remarkable fact is that the coefficients $L^N_{N-1}(a,b;\tau)$ are nonnegative not only for special values $\tau=\frac12, 1,2$ corresponding to spherical functions, but also for any $\tau>0$. This implies that the matrices $L^N_{N-1}$ are stochastic matrices for any $\tau>0$ (the fact that the row sums equal $1$ is obvious). 

In this way we obtain a projective chain with the states $S_N=\Sign(N)$, $N=1,2,\dots$, and parameter $\tau>0$; let us call it the \emph{Jack projective chain}.  Its boundary was described in Okounkov--Olshanski \cite{OO-IMRN}. 

For the remaining $10-3=7$ series of symmetric spaces, the spherical functions are expressed through  Heckman--Opdam's multivariate Jacobi polynomials  (see e.g. Heckman's lecture notes in \cite{Heckman}) with special values of parameters, and the definition \eqref{eq1.D}  can again be extrapolated to general values of parameters. The corresponding boundary problem was solved in Okounkov--Olshanski \cite{OO-2006}.

\subsection{The Macdonald chain}\label{sect1.4}

There is a natural extension of \eqref{eq1.E}, where the Jack polynomials are replaced by the Macdonald polynomials and the specialization at $(1^N)$ is replaced by that at $(1,t^{-1},\dots,t^{1-N})$: 
\begin{multline}\label{eq1.F}
\left.\frac{P_a(u_1,\dots,u_N;q,t)}{P_a(1,t^{-1},\dots,t^{1-N};q,t)}\right|_{u_N=t^{1-N}}\\
=\sum_{b\in\Sign(N-1)} L^N_{N-1}(a,b;q,t)\frac{P_b(u_1,\dots,u_{N-1};q,t)}{P_b(1,t^{-1},\dots,t^{2-N};q,t)}, \qquad a\in\Sign(N).
\end{multline}

An explicit expression for the coefficients $L^N_{N-1}(a,b;q,t)$ is provided by the \emph{branching rule for Macdonald polynomials} \cite[ch. VI, (7.14$'$)]{M}. In particular, the coefficients vanish unless the signatures $a$ and $b$ \emph{interlace}, meaning that
$$
a_i\ge b_i\ge a_{i+1}, \qquad 1\le i\le N-1.
$$
In this case we write $b\prec a$ or $a\succ b$.  
(Macdonald deals with ordinary (not Laurent) polynomials, which are indexed by partitions, but extending the results that we need to the more general case of Laurent polynomials, which are indexed by signatures, presents no difficulty.)

Suppose $0<q<1$, $0<t<1$ (equally well one could take $q>1$, $t>1$); then the coefficients $L^N_{N-1}(a,b;q,t)$ are nonnegative. Next, from \eqref{eq1.F} it is evident that 
$$
\sum_{b\in\Sign(N-1)}L^N_{N-1}(a,b;q,t)=1, \qquad \forall a\in\Sign(N).
$$
It follows that the matrices $L^N_{N-1}(q,t)$ with the entries $L^N_{N-1}(a,b;q,t)$ are stochastic matrices. In this way we obtain a  projective chain,
\begin{equation}\label{eq1.L}
\Sign(1)\stackrel{L^2_1(q,t)}{\dashleftarrow}\Sign(2)\stackrel{L^3_2(q,t)}{\dashleftarrow}\Sign(3)\stackrel{L^4_3(q,t)}{\dashleftarrow}\cdots, 
\end{equation}
which we call the \emph{Macdonald chain}; it depends on the two Macdonald parameters, $q$ and $t$, ranging over the open interval $(0,1)$.

Because of the limit relation 
$$
P_a(u_1,\dots,u_N;\tau)=\lim_{q\to 1}P_a(u_1,\dots,u_N;q,q^\tau)
$$
(\cite[ch. VI, \S10]{M}) we have  
$$
L^N_{N-1}(a,b;\tau)=\lim_{q\to1} L^N_{N-1}(a,b;q,q^\tau),
$$
so that the stochastic matrices defined by \eqref{eq1.E} are a limit case of the stochastic matrices defined by \eqref{eq1.F}. 

The first work related to the Macdonald chain was that of  Gorin \cite{G-AM}. He examined the special case of equal parameters, $q=t$, and obtained (among other things) the description of the boundary. In the case $q=t$ the Macdonald polynomials become the Schur polynomials, as do the Jack polynomials with $\tau=1$. However, as shown in \cite{G-AM}, the replacement of the specialization at $(1,\dots,1)$ by that at $(1,q^{-1},\dots,q^{1-N})$ drastically changes the structure of the boundary: it becomes a totally disconnected space.

Then Cuenca \cite{Cuenca} described the boundary of the Macdonald chain in a more general case, for   $t=q^\tau$, where $\tau=1,2,3,\dots$\,.

In another direction, Sato \cite{Sato1}, \cite{Sato2} linked Gorin's results to characters of a quantum version of the group $U(\infty)$.

\subsection{Motivation for further generalization}\label{sect1.5}

The approach of the present work allows to describe the boundary of the Macdonald chain \eqref{eq1.L} for arbitrary values $q,t\in(0,1)$. However, this is only a side result, as our concern is to study an \emph{extension} of the chain \eqref{eq1.L}. 

The need of such an extension was explained in the joint work by Gorin  and the author \cite{GO-2016} about a $q$-version of the so called \emph{zw-measures}. In the initial version, the zw-measures arose from the problem of harmonic analysis on the infinite-dimensional unitary group \cite{Ols-2003}, \cite{BO-2005}; they form a four-parameter family of probability measures on the boundary of the Jack chain with $\tau=1$. The Jack deformation of the zw-measures was constructed in \cite{Ols-2003a}, so it was natural to ask if there exists a $q$-deformation, too. 

Initially we tried to construct $q$-deformed zw-measures in the framework of Gorin's paper \cite{G-AM}, but that attempt failed. Then we understood the reason: the desired result can be achieved only after enlarging the sets $\Sign(N)$.   

We proceed to necessary definitions.

\subsection{Double signatures and extended Gelfand Tsetlin graph}

\begin{definition}\label{def1.DSign}
By a \emph{double signature of length $N$} we mean an ordered pair  of signatures $(a^+,a^-)$ such that $a^+\in\Sign(k)$, $a^-\in\Sign(l)$, and  $k+l=N$. The set of all such pairs is denoted by $\DSign(N)$. We do not exclude the case when $k$ or $l$ equals $0$. Thus, $\DSign(N)$ is the disjoint union of the sets
$$
\Sign(N)\times\{\varnothing\}, \quad \Sign(N-1)\times\Sign(1),\; \dots,\; \Sign(1)\times\Sign(N-1), \quad \{\varnothing\}\times \Sign(N),
$$
where $\{\varnothing\}$ is interpreted as a singleton (the `empty signature'). 
\end{definition}

We identify $\Sign(N)$ with the subset $\Sign(N)\times\{\varnothing\}\subset \DSign(N)$. Thus, we may regard $\DSign(N)$ as an extension of $\Sign(N)$.

\begin{definition}\label{def1.interlace}
We say that the double signatures $(a^+,a^-)\in\DSign(N)$ and $(b^+,b^-)\in\DSign(N-1)$ \emph{interlace} if one of the following two conditions holds:

(i) $(a^+,a^-)\in\Sign(k)\times \Sign(l)$ with $k>0$, $(b^+,b^-)\in\Sign(k-1)\times \Sign(l)$, and
$$
a^+_1\ge b^+_1\ge\dots \ge a^+_{k-1}\ge b^+_{k-1}\ge a^+_{k}, \qquad a^-_1\ge b^-_1\ge\dots \ge a^-_{l}\ge b^-_{l};
$$

(ii) $(a^+,a^-)\in\Sign(k)\times \Sign(l)$ with $l>0$, $(b^+,b^-)\in\Sign(k)\times \Sign(l-1)$, and
$$
a^+_1\ge b^+_1\ge\dots \ge a^+_k\ge b^+_k, \qquad a^-_1\ge b^-_1\ge\dots \ge a^-_{l-1}\ge b^-_{l-1}\ge a^-_l.
$$
\end{definition}

Equivalently: $b^+\prec a^+$ and $b^-\prec a^-$ with the understanding that if $b^\pm$ has the same length as $a^\pm$, then  $a^\pm$ should be replaced by $a^\pm\cup\{-\infty\}$.  

We write the interlacement relation for double signatures as $(b^+,b^-)\prec (a^+,a^-)$ or $(a^+,a^-)\succ (b^+,b^-)$

Recall that the \emph{Gelfand--Tsetlin graph} is the graded graph whose vertex set is the disjoint union $\bigsqcup_{N=1}^\infty \Sign(N)$ and the edges are formed by the pairs $b\prec a$ of interlacing signatures.

\begin{definition}\label{def1.extGT}
The \emph{extended Gelfand--Tsetlin graph} is the graded graph whose vertex set is the disjoint union $\bigsqcup_{N=1}^\infty \DSign(N)$ and the edges are formed by the pairs $(b^+,b^-)\prec (a^+,a^-)$ of interlacing double signatures. 
\end{definition}

This definition is equivalent to the one in \cite{GO-2016}. 

The embedding $\Sign(N)\to \DSign(N)$ via the map $a\mapsto (a,\varnothing)$ induces an embedding  of the conventional Gelfand--Tsetlin graph into the extended Gelfand--Tsetlin graph. 

Note that if both components $a^+$ and $a^-$ are nonempty, then there are infinitely many vertices $(b^+,b^-)\prec(a^+,a^-)$.  This is in sharp contrast with the conventional Gelfand--Tsetlin graph. 

Note also that the extended Gelfand-Tsetlin graph \emph{is not} the product of two copies of the conventional Gelfand-Tsetlin graph.

\subsection{Point configurations attached to double signatures}

\begin{definition}\label{def1.C}
We fix two parameters $q,t\in(0,1)$ and two additional parameters $\zeta_+>0$ and $\zeta_-<0$. To an arbitrary double signature $(a^+,a^-)\in\DSign(N)$, where $a^+\in\Sign(k)$, $a^-\in\Sign(l)$, $k+l=N$,  we assign an $N$-point configuration $X_N(a)=X_N(a^+,a^-)\subset\R^*$, as follows:
\begin{equation}\label{eq1.G}
X_N(a):=\{\zeta_+q^{-a^+_i}t^{i-1}: i=1,\dots,k\}\cup\{\zeta_-q^{-a^-_i}t^{i-1}: i=1,\dots,l\}.
\end{equation}
For each $N=1,2,3\dots$, we denote by $\Om_N$ the set of all configurations of the form \eqref{eq1.G} with $k+l=N$. 
\end{definition}

For instance, let $N=5$, $k=3$, $l=2$, $a^+=(5,3,1)$, and $a^-=(4,2)$. Then the configuration \eqref{eq1.G} is the set
$$
\{\zeta_+ q^{-5}, \; \zeta_+q^{-3}t, \; \zeta_+q^{-1}t^2\}\cup\{\zeta_-q^{-4}, \; \zeta_- q^{-2}t\}.
$$
Or, listing the points in the ascending order,
$$
\{\zeta_-q^{-4}, \; \zeta_- q^{-2}t, \; \zeta_+q^{-1}t^2,\; \zeta_+q^{-3}t, \; \zeta_+ q^{-5}\}.
$$

Let us emphasize that the definition of $X_N(a)$ and $\Om_N$ depends on the quadruple of parameters $(q,t,\zeta_+,\zeta_-)$, but we suppress them from the notation for the sake of brevity.

Note that distinct double signatures produce distinct configurations. This will enable us to switch from double signatures to point configurations and back. 

\subsection{The special case $t=q^\tau$, $\tau=1,2,3,\dots$}

In this case the description of $\Om_N$ simplifies. 

Namely, consider the  \emph{two-sided $q$-lattice}
\begin{equation}\label{eq1.H}
\L:=\{\zeta_+ q^n: n\in\Z\}\cup\{\zeta_- q^n: n\in\Z\}\subset\R^*.
\end{equation} 
If $t=q$, then the configurations $X\in\Om_N$ are precisely the $N$-point subsets of $\L$. 
Next, if $t=q^\tau$ with $\tau=2,3,\dots$, then the configurations $X\in\Om_N$ are the $N$-point subsets of $\L$ subject to the following constraint: between any two neighboring points of $X$ there are at least $\tau-1$ unoccupied nodes of the lattice. 

\subsection{Construction of extended Macdonald chain (Theorem A)}\label{results1}

Let $\Y$ denote the set of all partitions; as in \cite{M}, we identify partitions with the corresponding Young diagrams. For $N=1,2,\dots$, let  $\Y(N)\subset\Y$ denote the subset of partitions of length at most $N$. We have
$$
\Y(1)\subset \Y(2)\subset\Y(3)\subset\dots, \qquad \bigcup_{N=1}^\infty\Y(N)=\Y.
$$
Note that $\Y(N)$  may be viewed as a subset of $\Sign(N)$. 

We denote by $P_{\nu\mid N}(x_1,\dots,x_N;q,t)$ the $N$-variate Macdonald polynomial with parameters $(q,t)$ and the index $\nu\in\Y(N)$. Every symmetric polynomial $f$ in $N$ variables may be viewed as a function $f(X)$ on $\Om_N$; in particular, the function corresponding to a Macdonald polynomial is written as $P_{\nu\mid N}(X;q,t)$, where $X\in\Om_N$. We denote by $[X]$ the smallest closed interval of $\R$ containing $X$.  

For $z\in\C$ and a partition (=Young diagram) $\nu$ we set
\begin{equation}\label{eq1.J}
(z;q,t)_\nu:=\prod_{(i,j)\in\nu}(1-zq^{j-1}t^{1-i}),
\end{equation}
where $(i,j)$ denotes the box on the intersection of the $i$th row and $j$th column.

Our first main result is 

\begin{theoremA}
Let $q,t\in(0,1)$ and  $\zeta_+>0$, $\zeta_-<0$ be fixed. 
For each $N\ge2$ there exists a unique stochastic matrix $\La^N_{N-1}$ of format $\DSign(N)\times \DSign(N-1)$,  such that the following two conditions hold.

{\rm(i)} The entries $\La^N_{N-1}(a,b)$ of $\LaN$ are strictly positive if $b\prec a$ and are equal to\/ $0$ otherwise.

{\rm(ii)} For any $a\in\DSign(N)$ and any $\nu\in\Y(N-1)$,
\begin{equation}\label{eq1.I}
\sum_{b\in\DSign(N-1)}\LaN(a,b) \frac{P_{\nu\mid N-1}(X_{N-1}(b))}{(t^{N-1};q,t)_\nu}=\frac{P_{\nu\mid N}(X_N(a))}{(t^N;q,t)_\nu}.
\end{equation}
\end{theoremA}

\noindent\emph{Comments} 1.  An explicit expression for the matrix entries $\LaN(a,b)$ is given in subsection \ref{sect5.10}. Like $X_N(a)$, it depends on $(q,t,\zeta_+,\zeta_-)$, but we suppress these parameters from the notation for the sake of brevity. 
\smallskip

2. The matrix $\LaN$ extends the matrix $L^N_{N-1}(q,t)$  from \eqref{eq1.F},  in the sense that $L^N_{N-1}(q,t)$ can be identified with a submatrix (a diagonal block) of the matrix $\LaN$. This follows from the computation in Section \ref{sect2}. Note that the claim is not evident from the comparison of \eqref{eq1.I} with \eqref{eq1.F}: although both conditions are written in terms of Macdonald polynomials, they look quite different.
\smallskip

3. As pointed out in subsection \ref{sect1.4} above, the fact that the coefficients in the expansion \eqref{eq1.F} form a stochastic matrix immediately follows from the branching rule for the Macdonald polynomials.  In the context of Theorem A, the definition of $\LaN$ relies on Okounkov's $q$-integral representation for Macdonald polynomials \cite{Ok-CM}, and the proof of the theorem that we can offer is rather long. It can be simplified when $t=q^\tau$ with $\tau=1,2,3,\dots$; this case is examined separately in Section \ref{sect4}, so that the reader may skip section \ref{sect5} if desired. In the special case $q=t$, considered in \cite{GO-2016}, the proof is direct and easy (\cite[Proposition 2.4]{GO-2016}), due to the fact that then the entries of $\LaN$ are given by a simple formula. However, the approach of \cite{GO-2016} does not apply in the two-parameter case. 
\smallskip

4. We call \eqref{eq1.I} the \emph{coherency relation} for Macdonald polynomials.

\begin{definition}\label{def1.B} 
Theorem A allows us to form the Markov chain
\begin{equation}\label{eq1.K}
\DSign(1)\stackrel{\La^2_1}{\dashleftarrow}\DSign(2)\stackrel{\La^3_2}{\dashleftarrow}\DSign(3)\stackrel{\La^4_3}{\dashleftarrow}\dots\,.
\end{equation}
We call it the \emph{extended Markov chain}. 
\end{definition}

As is seen from claim (i) of the theorem, the transition probabilities of this Markov chain are attached to the edges of the extended Gelfand--Tsetlin graph.

\subsection{The boundary of the extended Macdonald chain (Theorem B)}\label{results2}
Our purpose is to find the boundary of the extended Macdonald chain \eqref{eq1.K} in the sense of Definition \ref{def7.A}. Informally, we call it the `$(q,t)$-boundary of the extended Gelfand--Tsetlin graph'. 

\begin{definition}
(i) By an \emph{infinite signature} we mean an arbitrary infinite sequence of integers $a_1\ge a_2\ge\dots$\,. The \emph{infinite double signature} is a pair $(a^+,a^-)$ of signatures of which at least one is infinite. The set of infinite double signatures will be denoted by $\DSign(\infty)$. 

(ii) To each infinite double signature $a=(a^+,a^-)$ one assigns an infinite point configuration $X_\infty(a)\subset\R^*$: it is defined as in \eqref{eq1.G}: the only change is that one of the indices $i, j$ (or both) will range over the whole set $\{1,2,3,\dots\}$.  
\end{definition}

Note that $\DSign(\infty)$ has the power of the continuum. The following is our second main result. 

\begin{theoremB}
The boundary of the extended Macdonald chain can be identified, in a natural way, with the set\/ $\DSign(\infty)$ of infinite double signatures. Under this identification, the Markov kernels $\La^\infty_N:\DSign(\infty)\dasharrow \DSign(N)$ can be characterized via the relations 
\begin{equation}\label{eq1.I1}
\sum_{b\in\DSign(N)}\La^\infty_N(a,b) \frac{P_{\nu\mid N}(X_N(b);q,t)}{(t^N;q,t)_\nu}=P_\nu(X_\infty(a);q,t), \qquad \forall\nu\in\Y(N),
\end{equation}
where $P_\nu(\ccdot;q,t)$ is the Macdonald symmetric function indexed by $\nu$. 
\end{theoremB}

\noindent\emph{Comments} 1. A bit more detailed form of this statement is given in Theorem \ref{thm8.B}.
\smallskip

2. The parametrization of the boundary does not depend on the parameters $(q,t, \zeta_+,\zeta_-)$ but the Markov kernels $\La^\infty_N$ do. 
\smallskip

3.  In the special case $t=q$ the result of Theorem B was proved in Gorin--Olshanski \cite[Theorem 3.12]{GO-2016}, 

4.  As a corollary of Theorem B, one can obtain a description of the boundary for the (not extended) Macdonald chain \eqref{eq1.L} with arbitrary $q,t\in(0,1)$. To do this it suffices to use the computation of section \ref{sect2}, see comment 2 to Theorem A above. This extends earlier results of Gorin \cite{G-AM} (for the case $t=q$) and of Cuenca \cite{Cuenca} (for $t=q^\tau$, $\tau=1,2,3,\dots$). Our approach is different from those of these works.

\subsection{Coherent systems of measures and the approximation theorem (Theorem C)}\label{results3}

\begin{definition}
A sequence $\{M_N: N=1,2,3,\dots\}$ of probability measures on the sets $\DSign(N)$, $N=1,2,3,\dots$, is said to be a \emph{coherent system of measures} if  for any $N\ge2$
\begin{equation}
\sum_{a\in\DSign(N)}M_N(a)\LaN(a,b)=M_{N-1}(b), \qquad \forall\, b\in\DSign(N-1).
\end{equation}
Or, in short form, $M_N\LaN=M_{N-1}$, where $M_N$ and $M_{N-1}$ are treated as row vectors. 
\end{definition}

By the very definition of the boundary, every coherent system $\{M_N\}$ gives rise to a probability measure $M_\infty$ on $\DSign(\infty)$, uniquely determined by the relations 
$$
M_\infty\La^\infty_N=M_N, \qquad N=1,2,\dots,
$$ 
where the left-hand side is the pushforward of $M_\infty$ by the Markov kernel $\La^\infty_N$ linking $\DSign(\infty)$ with $\DSign(N)$. We call $M_\infty$ the \emph{boundary measure} of the system $\{M_N\}$.

Consider the disjoint union 
$$
\wt{\DSign}:=\DSign(\infty) \sqcup \bigsqcup_{N=0}^\infty \DSign(N),
$$
where $\DSign(0)$ is a singleton, interpreted as the pair of empty signatures. We equip $\wt{\DSign}$ with a topology in the following way. 

\begin{definition}
Let us say that two signatures (finite or infinite, no matter) are \emph{$\eps$-close} (where $\eps>0$ is small), if they have the same set of coordinates exceeding $-\eps^{-1}$. Likewise, we say that two double signatures, $(a^+,a^-)$ and $(b^+,b^-)$, are \emph{$\eps$-close} if so are $a^\pm$ and $b^\pm$.
\end{definition}

The notion of $\eps$-closeness just defined makes $\wt{\DSign}$ a uniform space, and hence a topological space. It is a non-discrete locally compact space in which both $\DSign(\infty)$ and $\bigsqcup_{N=0}^\infty \DSign(N)$ are dense subsets. 

The definition of the boundary measure can be made more concrete due to the next theorem, which is our third main result. We call it the \emph{approximation theorem}.

\begin{theoremC}
For any coherent system $\{M_N\}$, the measures $M_N$ converge to the boundary measure $M_\infty$  in the weak topology of measures on the space $\wt{\DSign}$. 
\end{theoremC}

\subsection{Hypergeometric point processes}
The main result of Gorin--Olshanski \cite{GO-2016} was the construction of the so called $q$-zw-measures --- a family of probability measures on the space $\DSign(\infty)$, depending on the parameter $q$ and providing a $q$-version of the spectral measures coming from harmonic analysis on $U(\infty)$ \cite{BO-2005}. The $q$-zw-measures are further studied  in \cite{CGO}. 

In the subsequent paper \cite{Ols-MacdonaldTwo} it is shown that the construction of the $q$-zw-measures can be extended by adding the second Macdonald parameter, $t$. This is achieved based on Theorems A, B, C and leads  to a new family of random point processes. 

\subsection{Organization of the paper}

Section \ref{sect2} establishes the link with \cite{GO-2016}. In section \ref{sect3} we state Okounkov's theorem about the $q$-integral representation of Macdonald polynomials in a form convenient for later use. In section \ref{sect4} we prove Theorem A for the case $t=q^\tau$ with $\tau\in\{1,2,3,\dots\}$. Section \ref{sect5} gives the proof of Theorem A for arbitrary $q,t\in(0,1)$. The essence of the argument is a delicate limit transition in Okounkov's formula. The section ends with remarks concerning a continuous analogue of Theorem A.  Theorems B and C are proved in section \ref{sect8} after a  preparation occupying sections \ref{sect6}-\ref{sect7}.

\section{Link between two families of stochastic matrices}\label{sect2}

The purpose of this short section is to justify the claim in comment 2 to Theorem A (subsection \ref{results1} above). Namely, in Proposition \ref{prop2.A} below we show how to transform the equations  \eqref{eq1.F} to the same form as in \eqref{eq1.I}. This result  establishes a link between our setup and that of Gorin \cite{G-AM} and Cuenca \cite{Cuenca}. 

Note that the summation in \eqref{eq1.F} is actually taken over those signatures $b$ that interlace with $a$ (see \eqref{eq1.I}): indeed, this follows from the branching rule for Macdonald polynomials. Recall that the interlacement relation is denoted as $b\prec a$. 

To simplify the notation we assume here that $\zeta_+=1$. 
Given $a\in\Sign(N)$, we set 
$$
X_N(a):=X_N(a,\varnothing)=\{q^{-a_i}t^{i-1}: i=1,\dots,N\}. 
$$ 

Let $a\mapsto a^*$ denote the involutive map $\Sign(N)\to\Sign(N)$ defined by
$$
(a_1,\dots,a_N)\mapsto(-a_N,\dots,-a_1).
$$ 

\begin{proposition}\label{prop2.A}
Let the $L^N_{N-1}(a,b;q,t)$ be the coefficients from the expansion \eqref{eq1.F}. For any $N=2,3,\dots$, signature $a\in\Sign(N)$, and partition $\nu\in\Y(N-1)$ we have
\begin{equation}\label{eq2.B}
\sum_{b\prec a} L^N_{N-1}(a^*,b^*;q,t) \frac{P_{\nu\mid N-1}(X_{N-1}(b);q,t)}{(t^{N-1};q,t)_\nu}=\frac{P_{\nu\mid N}(X_N(a);q,t)}{(t^N;q,t)_\nu}.
\end{equation}
\end{proposition}

\begin{proof}
Using the fact that Macdonald polynomials are homogeneous we rewrite \eqref{eq1.F} as 
\begin{multline}\label{eq2.C}
\frac{P_{a\mid N}(u_1t^{N-1},\dots,u_{N-1}t^{N-1},1;q,t)}{P_{a\mid N}(t^{N-1}, t^{N-2},\dots, 1;q,t)}\\
=\sum_{b\prec a} L^N_{N-1}(a,b;q,t)\frac{P_{b\mid N-1}(u_1t^{N-2},\dots,u_{N-1}t^{N-2};q,t)}{P_{b\mid N-1}(t^{N-2},t^{N-3},\dots,1;q,t)}.
\end{multline} 
Take $\nu\in\Y(N-1)$ and substitute in \eqref{eq2.C}
$$
(u_1,\dots,u_{N-1})=(q^{\nu_1},\, q^{\nu_2}t^{-1}, \, \dots\, q^{\nu_{N-1}}t^{2-N}).
$$
Then we obtain
\begin{multline}\label{eq2.D}
\frac{P_{a\mid N}(q^{\nu_1}t^{N-1},\dots,q^{\nu_{N-1}}t, 1;q,t)}{P_{a\mid N}(t^{N-1}, t^{N-2},\dots, 1;q,t)}\\
=\sum_{b\prec a} L^N_{N-1}(a,b;qt)\frac{P_{b\mid N-1}(q^{\nu_1}t^{N-2},\dots,q^{\nu_{N-1}};q,t)}{P_{b\mid N-1}(t^{N-2},t^{N-3},\dots,1;q,t)}.
\end{multline} 

Next, recall the \emph{label-argument symmetry relation} for Macdonald polynomials (\cite[ch. VI, (6.6)]{M}):
\begin{equation}\label{eq2.sym}
\frac{P_{\mu\mid N}(q^{\la_1}t^{N-1},\dots,q^{\la_{N-1}}t, q^{\la_N};q,t)}{P_{\mu\mid N}(t^{N-1}, t^{N-2},\dots, 1;q,t)}=\frac{P_{\la\mid N}(q^{\mu_1}t^{N-1},\dots,q^{\mu_{N-1}}t, q^{\mu_N};q,t)}{P_{\la\mid N}(t^{N-1}, t^{N-2},\dots, 1;q,t)}.
\end{equation}
In \cite{M}, this relation is established for ordinary (non-Laurent) Macdonald polynomials, so that $\la$, $\mu$ are assumed to be partitions of length at most $N$. But the result is immediately extended to the Laurent version of Macdonald polynomials labelled by signatures. Applying the symmetry relation to both sides of \eqref{eq2.D} we obtain  
\begin{multline}\label{eq2.E}
\frac{P_{\nu\mid N}(q^{a_1}t^{N-1},\dots,q^{a_{N-1}}t, q^{a_N};q,t)}{P_{\nu\mid N}(t^{N-1}, t^{N-2},\dots, 1;q,t)}\\
=\sum_{b\prec a} L^N_{N-1}(a,b;q,t)\frac{P_{\nu\mid N-1}(q^{b_1}t^{N-2},\dots,q^{b_{N-1}};q,t)}{P_{\nu\mid N-1}(t^{N-2},t^{N-3},\dots,1;q,t)}.
\end{multline} 
Then we replace $(a,b)$ with $(a^*,b^*)$, which gives
\begin{equation}\label{eq2.F}
\frac{P_{\nu\mid N}(X_N(a);q,t)}{P_{\nu\mid N}(t^{N-1}, t^{N-2},\dots, 1;q,t)}
=\sum_{b\prec a} L^N_{N-1}(a^*,b^*;q,t)\frac{P_{\nu\mid N-1}(X_{N-1}(b);q,t)}{P_{\nu\mid N-1}(t^{N-2},t^{N-3},\dots,1;q,t)}.
\end{equation}

Finally, from the principal specialization formula \cite[Ch VI, (6.11$'$)]{M} it follows that $P_{\nu\mid N}(1,t,\dots,t^{N-1};q,t)$ differs from $(t^N;q,t)_\nu$ by a factor that depends only on $\nu$. Therefore, \eqref{eq2.F} implies \eqref{eq2.B}.
\end{proof}

 A similar computation is contained in Appendix B of Forrester--Rains \cite{FR}. See also the end of \S2 of \cite{GO-2016}.

\section{Okounkov's $q$-integral formula}\label{sect3}

In this section we formulate a result due to Okounkov \cite{Ok-CM}, which is substantially used in the sequel.

\subsection{Notation from $q$-calculus}

Throughout the paper we use the standard notation from $q$-calculus (see Gasper--Rahman \cite{GR}):
For $z\in\C$ and $m=0,1,2,\dots$
$$
(z;q)_\infty:=\prod_{n=0}^\infty(1-zq^n), \quad (z;q)_m:=\prod_{n=0}^{m-1}(1-zq^n)=\frac{(z;q)_\infty}{(zq^m;q)_\infty}.
$$
The $q$-integral of a function $f$ in the complex domain is defined by (below $z,z'\in\C$)
\begin{equation}\label{eq3.E}
\begin{gathered}
\int_{z'}^z f(w)d_q w:=\int_0^z f(w)d_q w-\int_0^{z'}f(w)d_q w, \\
\text{where} \quad \int_0^z f(w)d_q w:=(1-q)\sum_{n=0}^\infty f(zq^n) zq^n.
\end{gathered}
\end{equation}
These definitions make sense for any complex $q$ with $|q|<1$.

\subsection{Okounkov's formula}\label{sect3.2}
Introduce the following meromophic function in $2N-1$ variables:  
\begin{multline}\label{eq3.A}
 R(z_1,\dots,z_N;w_1,\dots,w_{N-1};q,t)\\
 :=\prod_{r=1}^{N-1}\prod_{s=1}^N\frac{(w_rq/z_s;q)_\infty}{(w_rt/z_s;q)_\infty} \prod_{1\le i\ne j\le N}\frac{(z_it/z_j;q)_\infty}{(z_iq/z_j;q)_\infty}.
\end{multline}
Below we write $Z=(z_1,\dots,z_N)$, $W=(w_1,\dots,w_{N-1})$  and set
$$
V(Z):= \prod\limits_{1\le i<j\le N}(z_i-z_j), \quad V(W):=\prod\limits_{1\le r<s\le N-1}(w_r-w_s). 
$$
A constant $C_N(q,t)$ is defined by
\begin{equation}\label{eq3.D}
C_N(q,t):=\frac{((t;q)_\infty)^N}{(1-q)^{N-1}(t^N;q)_\infty((q;q)_\infty)^{N-1}}.
\end{equation}
Recall that the symbol $(z;q,t)_\nu$ was defined in \eqref{eq1.J}.

Finally, recall that $P_{\nu\mid N}(\ccdot;q,t)$ is our notation for the $N$-variate Macdonald polynomial indexed by a partition $\nu\in\Y(N)$, with parameters $(q,t)$. Our normalization of these polynomials is the same as in \cite{M}; that is, the monomial $x_1^{\nu_1}\dots x_N^{\nu_N}$ enters $P_{\nu\mid N}(x_1,\dots,x_N;q,t)$ with coefficient $1$.

\begin{theorem}[Okounkov \cite{Ok-CM}, Theorem I]\label{thm3.A}
Let $q\in\C$,  $|q|<1$. Next, let $N=2,3,\dots$ and $\nu\in\Y(N-1)$. Finally, let  $Z=(z_1,\dots,z_N)\in\C^N$ and $t\in\C$ be  in general position.

Then the following formula holds
\begin{multline}\label{eq3.B}
\frac{C_N(q,t)}{V(Z)}\,\int^{z_1}_{z_2}d_q w_1\int^{z_2}_{z_3}d_q w_2\dots \int^{z_{N-1}}_{z_N}d_q w_{N-1}\\
\times V(W)\,  R(z_1,\dots,z_N;w_1,\dots,w_{N-1};q,t)\frac{P_{\nu\mid N-1}(w_1,\dots,w_{N-1};q,t)}{(t^{N-1};q,t)_\nu}\\
=\frac{P_{\nu\mid N}(z_1,\dots,z_N;q,t)}{(t^N;q,t)_\nu}.
\end{multline}
\end{theorem}

Note that the assumption about general position (not explicitly mentioned in \cite{Ok-CM}) is imposed  in order to avoid possible singularities of the integrand on the $q$-contour of integration.

\subsection{A special case}

Suppose that $t=q^\tau$, where $\tau\in\{1,2,\dots\}$. Then the expression \eqref{eq3.A} can be simplified:

\begin{equation}\label{eq3.C}
 R(z_1,\dots,z_N;w_1,\dots,w_{N-1};q,q^\tau)=\dfrac{\prod\limits_{r=1}^{N-1}\prod\limits_{s=1}^N(w_rq/z_s;q)_{\tau-1}}{\prod\limits_{1\le i\ne j\le N}(z_iq/z_j;q)_{\tau-1}}.
\end{equation}
In the case $t=q$ we simply have
$$
 R(z_1,\dots,z_N;w_1,\dots,w_{N-1};q,q)\equiv1.
$$

\subsection{Another special case} 

Another kind of simplification in Theorem \ref{thm3.A} occurs when  $q=0$ (the case of Hall--Littlewood polynomials). Then \eqref{eq3.A} turns into
$$
 R(z_1,\dots,z_N;w_1,\dots,w_{N-1};0,t)\\
 :=\prod_{j=1}^N\dfrac{\prod_{i:\, i\ne j}(z_j-z_it)}{\prod_{i=1}^{N-1}(z_j-w_it)}
$$
and the $q$-integral \eqref{eq3.E} reduces to 
$$
\left.\int_{z'}^z f(w)d_qw \right|_{q=0}=z f(z)-z'f(z'),
$$
so that the multiple $q$-integral \eqref{eq3.B} reduces to a finite sum (because of the factor $V(W)$, the sum actually comprises not $2^{N-1}$ but only $N$ summands).

\subsection{Remark on alternate derivation of Okounkov's formula}

Okounkov first shows (\cite[Proposition 3.4]{Ok-CM}) that the left-hand side of \eqref{eq3.B}  is a symmetric polynomial in variables $(z_1,\dots,z_N)$, of the same degree (in this claim the polynomial $P_{\nu\mid N-1}(w_1,\dots,w_{N-1};q,t)$ can be replaced by an arbitrary symmetric polynomial). This part of his proof is relatively easy and short, while the remaining part is longer and more intricate. But the latter part can be replaced by the following argument. The aforementioned proposition allows to reduce the desired idenity \eqref{eq3.B} to the special case when, in the notation of section \ref{sect2}, 
$$
(z_1,\dots,z_N)=X_N(a), \qquad a\in\Sign(N).
$$
And then the computation of section \ref{sect2} allows to further reduce the identity to the branching rule for the Macdonald polynomials. 

After I realized this, I found a similar remark in Forrester--Rains \cite[comment after (3.17)]{FR}.

\section{Proof of Theorem A: special case $t=q^\tau$, $\tau=1,2,3,\dots$} \label{sect4}

\subsection{Preliminaries} 
 
Throughout this section we assume that $t=q^\tau$, where $0<q<1$ and $\tau$ is a fixed positive integer. 

Recall the definition \eqref{eq1.G}:  the \emph{two-sided $q$-lattice} $\L\subset\R\setminus\{0\}$ is the subset 
\begin{equation}\label{eq4.A}
\L=\L_+\sqcup\L_-, \qquad \L_\pm:=\{\zeta_\pm q^n: n\in\Z\}\subset\R^*,
\end{equation}
where $\zeta_-<0$ and $\zeta_+>0$ are two fixed extra parameters. We need these two parameters only to define the lattice $\L$. Note that $\L$ does not change if $\zeta_+$ or $\zeta_-$ is multiplied by an integral power of parameter $q$. Points of $\L$ are called \emph{nodes}. 

In this  section, by a  \emph{configuration} we always mean a subset $X\subset\L$.  
We say that a configuration $X$ is \emph{$\tau$-sparse} if any two distinct points of $X$ are separated by at least  $\tau-1$ empty nodes (of course, this is a real constraint only for $\tau\ge2$). 

Equivalently, $X$ is $\tau$-sparse if for any two distinct points $x, x'$ of $X$ the following condition holds: $0<x'<x$ implies $x'\le x t$, and $x<x'<0$ implies $x'\ge xt$. One more equivalent formulation: for any two distinct points $x, \,x'\in X$ of the same sign, one has $|\log_q(x/x')|\ge \tau$.

\begin{definition}\label{def4.B}
Given two nodes $x'<x$ of $\L$, which are separated by at least $\tau-1$ other nodes, we introduce a special notion of \emph{$q$-interval} $\Int_\tau(x',x)$. This is a set of nodes whose definition depends on the position of the pair $(x',x)$ with respect to $0$:

\smallskip

1. If $0<x'<x$, then $\Int_\tau(x',x):=\{y\in\L_+: x' q^{-\tau}\le y\le x\}$.

2. If $x'<x<0$, then $\Int_\tau(x',x):=\{y\in\L_-: x'\le y\le xq^{-\tau}\}$.

3. If $x'<0<x$, then $\Int_\tau(x',x):=\{y\in\L: x'\le y\le x\}$.

\end{definition}

Note that in the third case $\Int_\tau(x',x)$ contains infinitely many nodes. 

\begin{definition}\label{def4.C}
(i) For each $N=1,2,\dots$ we denote by $\Om_N$ the set of $\tau$-sparse $N$-point configurations on $\L$. 

(ii) We say that two configurations $X\in\Om_N$ and $Y\in\Om_{N-1}$ \emph{interlace} if the following condition holds. Write $X=(x_1>\dots>x_N)$ and $Y=(y_1>\dots>y_{N-1})$; then we require that $y_i\in\Int_\tau(x_{i+1},x_i)$ for every $i=1,\dots,N-1$.   
\end{definition}

This definition of $\Om_N$  in item (i) agrees with Definition \ref{def1.C}: the $N$-point $\tau$-sparse configurations are precisely the configurations $X_N(a^+,a^-)$ coming from double signatures $(a^+,a^-)\in\DSign(N)$. The definition of interlacement in item (ii) agrees with Definition \ref{def1.interlace}. Below we write the interlacement relation as $X\succ Y$ or $Y\prec X$. 

\subsection{The matrices $\LaN$}\label{sect4.B}

Let us agree to enumerate the points of a given configuration $X\in\Om_N$ in the descending order: $X=(x_1>\dots>x_N)$. Keeping this in mind, we set 
$$
V(X)=\prod_{1\le i<j\le N}(x_i-x_j).
$$
Thus, $V(X)>0$. In the case $N=1$ we agree that $V(X)=1$.

\begin{definition}\label{def4.A}
For each $N=2,3,\dots$ we define the matrix $\LaN$ of format $\Om_N\times\Om_{N-1}$ with the following entries $\LaN(X,Y)$:

$\bullet$ $\LaN(X,Y)=0$ unless $X\succ Y$.

$\bullet$ If $X\succ Y$, then 
\begin{equation}\label{eq4.B}
\La^N_{N-1}(X,Y)=\frac{((t;q)_\infty)^N}{(t^N;q)_\infty((q;q)_\infty)^{N-1}}\cdot\frac{V(Y)}{V(X)}\cdot \prod_{y\in Y}|y|\cdot
\frac{\prod\limits_{y\in Y}\prod\limits_{x\in X}(yq/x;q)_{\tau-1}}{\prod\limits_{x,x'\in X: \, x\ne x'}(xq/x';q)_{\tau-1}}.
\end{equation}
\end{definition}

In the simplest case $\tau=1$ the expression on the right-hand side simplifies and reduces to 
\begin{equation}\label{eq4.H}
(q;q)_{N-1} \prod\limits_{y\in Y}|y|\cdot\dfrac{V(Y)}{V(X)},
\end{equation}
which agrees with the definition given in \cite{GO-2016} and \cite{Ols-2016}.

\begin{lemma}\label{lemma4.A}
All the entries $\La^N_{N-1}(X,Y)$ are nonnegative. 
\end{lemma} 

\begin{proof}
We will show that $X\succ Y$ entails $\La^N_{N-1}(X,Y)>0$. In the case $\tau=1$ this follows immediately from \eqref{eq4.H}, so we will assume that $\tau\ge2$. 

Suppose first that $X\subset\L_+$, that is, all points of $X$ are on the right of $0$.  Then, according to Definition \ref{def4.C}, the following inequalities hold
\begin{gather}
x_1>\dots>x_N>0, \quad x_i\ge x_{i+1}q^{-\tau}, \qquad i=1,\dots,N-1; \label{eq4.C1}\\
y_1>\dots>y_{N-1}, \quad x_i\ge y_i\ge x_{i+1}q^{-\tau}\qquad i=1,\dots,N-1.\label{eq4.C2}
\end{gather}
We have to prove that the quantity  
$$
\frac{\prod\limits_{y\in Y}\prod\limits_{x\in X}(yq/x;q)_{\tau-1}}{\prod\limits_{x,x'\in X: \, x\ne x'}(xq/x';q)_{\tau-1}}
$$
is strictly positive. Let us split it in two parts:
$$
\frac{\prod\limits_{N-1\ge i\ge j\ge1}(y_iq/x_j;q)_{\tau-1}}{\prod\limits_{N\ge i>j\ge1}(x_iq/x_j;q)_{\tau-1}}\cdot \frac{\prod\limits_{1\le i< j\le N}(y_iq/x_j;q)_{\tau-1}}{\prod\limits_{1\le i<j\le N}(x_iq/x_j;q)_{\tau-1}}.
$$
In the first part, all the factors in the numerator and denominator are strictly positive because of  \eqref{eq4.C1} and \eqref{eq4.C2}: indeed, $i\ge j$ implies $0<y_iq/x_j<1$, while $i>j$ implies $0<x_iq/x_j<1$ (here we use the assumption $\tau\ge2$). 

Next, the second part of our expression can be rewritten in the form 
$$
\prod_{i=1}^{N-1}\prod_{j=i+1}^N\frac{(y_iq/x_j;q)_{\tau-1}}{(x_iq/x_j;q)_{\tau-1}}=\prod_{i=1}^{N-1}\prod_{j=i+1}^N\prod_{r=1}^{\tau-1}\frac{1-y_iq^r/x_j}{1-x_iq^r/x_j}.
$$
Now the inequalities \eqref{eq4.C1} and  \eqref{eq4.C2} show that all the factors in the denominator and the numerator are strictly negative. Since they contain equally many factors, we conclude that the whole expression is strictly positive. 

The same reasoning is applicable in the case $X\subset\L_-$ (observe that the definition of $\tau$-sparse configurations is symmetric with respect to reflection about zero, and so is the interlacement relation).  

Now we turn to the case when $0$ sits somewhere inside $X$, that is, $x_{k+1}<0<x_k$ for some $k<N$. Let us split $X$ into two parts,  $X_-\sqcup X_+$, where $X_-:=X\cap\L_-$  and $X_+:=X\cap\L_+$. Next, observe that $(yq/x;q)_{\tau-1}>0$ whenever $y$ and $x$ have opposite signs. Likewise,  
$(xq/x';q)_{\tau-1}>0$ whenever $x$ and $x'$ have opposite signs. Let us discard the corresponding factors.

Note that $y_k\in[x_{k+1},x_k]$ and all factors $(y_kq/x;q)_{\tau-1}$, where $x$ has the same sign as $y_k$, are strictly positive, because $y_kq/x<1$. Thus, these factors may again be discarded. After that the problem is reduced to the case when $X=X_\pm$, examined above. 
\end{proof}

The next remark will be used in the proof of Theorem \ref{thm4.A}

\begin{remark}\label{rem4.A}
Let $X=(x_1>\dots>x_N)\in\Om_N$ and let $Y=(y_1>\dots>y_{N-1})$ be an $(N-1)$-point configuration such that $y_i\in\Int_1(x_{i+1},x_i)$ for $i=1,\dots,N-1$. Then the right-hand side of \eqref{eq4.B} vanishes unless the stronger condition $y_i\in\Int_\tau(x_{i+1},x_i)$ holds for all $i$, meaning that $Y$ must interlace with $X$ in the sense of Definition \ref{def4.C}. 

Indeed, we have to show that if $y_i\in\Int_1(x_{i+1}, x_i)\setminus\Int_\tau(x_{i+1}, x_i)$ for some $i$, then the right-hand side of \eqref{eq4.B} vanishes. Let us examine the possible cases.

(1)  $x_{i+1}<0<x_i$. Then $\Int_1(x_{i+1},x_i)=\Int_\tau(x_{i+1},x_i)$, so the claim is trivial.  

(2)  $0<x_{i+1}=x_i q^\ell <x_i$ with $\ell\in\{\tau, \tau+1,\tau+2,\dots\}$. Then  
\begin{gather*}
\Int_1(x_{i+1}, x_i)=\{y\in\L_+: x_{i+1}q^{-1}\le y\le x_i\}, \\ 
\Int_\tau(x_{i+1},x_i)=\{y\in\L_+: x_{i+1}q^{-\tau}\le y\le x_i\}. 
\end{gather*}
If $y_i\in\Int_1(x_{i+1}, x_i)\setminus\Int_\tau(x_{i+1}, x_i)$, then $(y_iq/x_{i+1};q)_{\tau-1}=0$, so that the right-hand side of \eqref{eq4.B} vanishes. 

 (3) $x_{i+1}<x_i=x_{i+1}q^\ell<0$ with $\ell\in\{\tau, \tau+1,\tau+2,\dots\}$. Then
\begin{gather*}
\Int_1(x_{i+1}, x_i)=\{y\in\L_-: x_{i+1}\le y\le x_i q^{-1}\},\\ 
\Int_\tau(x_{i+1},x_i)=\{y\in\L_-: x_{i+1}\le y\le x_iq^{-\tau}\}.
\end{gather*}
If $y_i\in\Int_1(x_{i+1}, x_i)\setminus\Int_\tau(x_{i+1}, x_i)$, then $(y_iq/x_{i};q)_{\tau-1}=0$, so that the right-hand side of \eqref{eq4.B} vanishes. 
\end{remark}

\subsection{The coherency relation}

Recall that $\Y$ denotes the set of partitions, which we identify with the corresponding Young diagrams. Next, the length of a partition $\nu\in\Y$ is denoted by $\ell(\nu)$, and  $\Y(N):=\{\nu\in\Y: \ell(\nu)\le N\}$. 

Let $\Sym(N)$ denote the algebra of symmetric polynomials with $N$ variables (as the base field one can take $\R$ or $\C$).  Every polynomial $f\in\Sym(N)$ gives rise to a function on $\Om_N$: if $X=(x_1,\dots,x_N)\in\Om_N$, then we write $f(X):=f(x_1,\dots,x_N)$ (since $f$ is symmetric, the result does not depend on the enumeration of the points of $X$).  In particular, this is applicable to Macdonald polynomials, and we write their values at configurations $X\in\Om_N$ as $P_{\nu\mid N}(X;q,t)$.

\begin{theorem}\label{thm4.A}
Let $N=2,3,\dots$ and $\LaN$ be the matrix of format\/ $\Om_N\times\Om_{N-1}$ introduced in Definition \ref{def4.A}. Recall that $t=q^\tau$, where $0<q<1$ and $\tau$ is a positive integer. For any $\nu\in\Y(N-1)$ and any $X\in\Om_N$ the following `coherency relation' holds 
\begin{equation}\label{eq4.D}
\sum_{Y\in\Om_{N-1}}\LaN(X,Y) \frac{P_{\nu\mid N-1}(Y;q,t)}{(t^{N-1};q,t)_\nu}=\frac{P_{\nu\mid N}(X;q,t)}{(t^N;q,t)_\nu}.
\end{equation}
\end{theorem}

Recall that the symbol $(z;q,t)_\nu$ was defined in \eqref{eq1.J}. Note that $(z;q,t)_\nu\ne0$ if $z$ is real and less than $1$, hence $(t^N;q,t)_\nu\ne0$ and $(t^{N-1};q,t)_\nu\ne0$ for any $N=2,3,\dots$\,.

\begin{corollary}\label{cor4.A}
The $\LaN$, $N=2,3,\dots$, are stochastic matrices. 
\end{corollary}

\begin{proof}[Proof of the corollary]
By Lemma \ref{lemma4.A}, the matrix entries are nonnegative. Next, take in \eqref{eq4.D} as $\nu$ the zero partition (= empty Young diagram). In this case the corresponding Macdonald polynomials are identically equal to $1$ and the generalized Pochhammer symbols in the denominators also equal  $1$. Then \eqref{eq4.D} means that the row sums of $\LaN$ equal $1$. We conclude that $\LaN$ is a stochastic matrix.
\end{proof}

\begin{proof}[Proof of the theorem]
We enumerate the points of $X$ and $Y$ in descending order: $X=(x_1>\dots>x_N)$, $Y=(y_1>\dots>y_{N-1})$. Then $X\succ Y$ means exactly that $y_i\in\Int_\tau(x_{i+1},x_i)$ for all $i=1,\dots,N-1$.

First of all, observe that the series on the left-hand side of \eqref{eq4.D} converges. Indeed, if $X$ is entirely contained in $\L_+$ or $\L_-$, then there are only finitely many $Y$'s interlacing with $X$, so the sum is finite. If $X$ has points both in $\L_+$ and $\L_-$, then there exists a unique index $k$ such that $x_{k+1}<0<x_k$. It follows that the $q$-interval $\Int_\tau(x_{k+1},x_k)$ comprises infinitely many nodes, while all other $q$-intervals $\Int_\tau(x_{i+1},x_i)$ contain finitely many nodes. Then the series on the left-hand side is infinite, but its convergence is assured by the factor $|y_k|$ entering the right-hand side of \eqref{eq4.B}. 

In the simplest case $\tau=1$, the Macdonald polynomials turn into the Schur functions, and then \eqref{eq4.D} admits a simple direct proof, see Gorin-Olshanski \cite[Proposition 2.7]{GO-2016} and  Kim-Stanton \cite{KS}. In the case $\tau=2,3,\dots$ we derive \eqref{eq4.D} from Okounkov's formula \eqref{eq3.B}.   

A simple but important observation is that if $x'<x$ is a pair of points of $\L$, then 
\begin{equation}\label{eq4.G}
\int_{x'}^x f(y)d_qy=(1-q)\sum_{y\in\Int_1(x',x)}|y|f(y).
\end{equation}
Indeed, if $x'<0<x$, this follows directly from the definition \eqref{eq3.E} of the $q$-integral. If $x$ and $x'$ are of the same sign, then the terms corresponding to points $y$ lying outside the set $\Int_1(x',x)\subset[x',x]$ cancel out. Let us emphasize that the assumption $x',x\in\L$ is crucial here; if it is dropped, then it may well happen that the $q$-integral depends on the values of $f$ at some points $y$ outside $[x',x]$.

Let us return to Okounkov's formula \eqref{eq3.B} and show that it reduces to \eqref{eq4.D} when $x_1,\dots,x_N$ are the points of a configuration $X\in\Om_N$ enumerated in the descending order. 

Indeed, due to \eqref{eq4.G},  the multiple $q$-integral in \eqref{eq3.B} becomes a sum over the set
$\Int_1(x_N,x_{N-1})\times \dots\times \Int_1(x_2,x_1)$. Let us compare this sum with the sum on the left-hand side of \eqref{eq4.D}. From comparison of \eqref{eq4.B} with \eqref{eq3.C} it is seen that both expressions have the same form, and the only apparent difference is that in \eqref{eq4.D}, the summation is taken over the smaller set $\Int_\tau(x_N,x_{N-1})\times\dots\times\Int_\tau(x_2,x_1)$.  
However, this does not matter because of Remark \ref{rem4.A}, which shows that all the extra summands actually vanish.

This completes the proof of the theorem.
\end{proof}

Theorem \ref{thm4.A} and Corollary \ref{cor4.A} provide a proof of Theorem A in the case of $t=q^\tau$, $\tau\in\{1,2,,\dots\}$, except the uniqueness claim. The proof of the latter claim (in the general case) in given below in section \ref{sect5.3}.

\section{Proof of Theorem A:  general case}\label{sect5}

Throughout this section $\zeta_+>0$ and $\zeta_-<0$ are fixed parameters; $q$ and $t$ are supposed to lie in the open interval $(0,1)$.

The goal of this section is to prove Theorem A (section \ref{results1}) and exhibit an explicit expression for the matrices $\LaN$ (see subsection \ref{sect5.10}).

In the course of the proof, we sometimes need to introduce the assumption that $t$ is in general position with respect to $q$, but each time this constraint is ultimately removed.

\subsection{The sets $\Om_N$}\label{sect5.1}

As in Section \ref{sect4}, it will be convenient for us to switch from double signatures $a\in\DSign(N)$ to the corresponding point configurations $X_N(a)\in\Om_N$.  We start with a direct description of the sets $\Om_N$ and then translate the notion of interlacement $b\prec a$ into the language of configurations. This looks a bit more complicated as in the special case examined in Section \ref{sect4}. 

Below we use the notation
$$
q^\Z:=\{q^m: m\in\Z\}, \quad \Z_{\ge0}:=\{0,1,2,\dots\}, \quad q^{\Z_{\ge0}}:=\{q^m: m\in\Z_{\ge0}\}.
$$
Given an $N$-tuple $X=(x_1>\dots>x_n)$ of nonzero real numbers, we set
$$
k=k(X):=\#\{i=1,\dots,N: x_i>0\}.
$$

\begin{lemma}\label{lemma5.Omega}
Let $N=1,2,\dots$\,. The set\/ $\Om_N$ introduced in Definition \ref{def1.C} consists of $N$-tuples $(x_1>\dots>x_N)$ of nonzero real numbers satisfying the following constraints {\rm(1) -- (4):}
\smallskip

{\rm(1)} If $k\ge1$, then $x_1\in \zeta_+ q^{\Z}$.

{\rm(2)} If $k\ge 2$, then $x_{i+1}\in x_iq^{\Z_{\ge0}}t$ for each $i=1,\dots,k-1$.

{\rm(3)} If $N-k\ge1$, then $x_N\in\zeta_- q^{\Z}$.

{\rm(4)} If $N-k\ge2$, then $x_{i-1}\in x_i q^{\Z_{\ge0}} t$ for each $i=N,\dots,k+2$.

\end{lemma}

\begin{proof}
Evident from Definition \ref{def1.C}.
\end{proof}

Let $X=(x_1>\dots>x_N)\in\Om_N$ and $Y=(y_1>\dots>y_{N-1})\in\Om_{N-1}$, where $N=2,3,\dots$\,. We say that $X$ and $Y$ \emph{interlace} (and then write $X\succ Y$ or $Y\prec X$) if the corresponding double signatures interlace in the sense of Definition \ref{def1.interlace}. 

\begin{lemma}\label{lemma5.interlace}
Let $X\in\Om_N$ and $Y\in\Om_{N-1}$. Set $k=k(X)$. The configurations $X$ and $Y$ interlace if and only if for each $r=1,\dots,N-1$, the coordinate $y_r$ satisfies the following condition varying depending on the relative position of $r$ and $k${\rm:}

{\rm(1)} Suppose $1\le r<k$ and observe that this implies that $0<x_{r+1}<x_r$ and $x_{r+1}=x_rq^{l_r} t$ with some $l_r\in\Z_{\ge0}$. Then $y_r=x_r q^{m_r}$ with $0\le m_r\le l_r$. 

{\rm(2)} Suppose $k<r\le N-1$ and observe that this implies that  $x_{r+1}<x_r<0$ and $x_r=x_{r+1}q^{l_r}t$ with some $l_r\in\Z_{\ge0}$. Then  $y_r=x_{r+1} q^{m_r} $ with $0\le m_r\le l_r$.

{\rm(3)} Suppose $0<k=r<N$ and observe that this implies $x_{k+1}<0<x_k$. Then $y_k=x_k q^{m_k}$ or $y_k=x_{k+1}q^{m_k}$, where $m_k\in\Z_{\ge0}$ may be arbitrary.
\end{lemma}

\begin{proof}
Evident from Definition \ref{def1.interlace}.
\end{proof}

Note that in all cases we have $x_{r+1}\le y_r\le x_r$. 

\subsection{Scheme of proof}\label{sect5.2}

The coherency relation \eqref{eq1.I} can now be rewritten in the form 
\begin{equation}\label{eq5.F}
\sum_{Y\in\Om_{N-1}}\LaN(X,Y)\frac{P_{\nu\mid N-1}(Y;q,t)}{(t^{N-1};q,t)_\nu}=\frac{P_{\nu\mid N}(X;q,t)}{(t^N;q,t)_\nu},
\end{equation}
where $X\in\Om_N$ and $\nu\in\Y(N-1)$ are arbitrary.

We are going to prove that \eqref{eq5.F} holds true for a certain matrix $\LaN$ of format $\Om_N\times\Om_{N-1}$ whose  entries $\LaN(X,Y)$ are strictly positive for $X\succ Y$ and equal $0$ otherwise. Once this is done, the remaining claims of Theorem A will follow quickly. Indeed, the fact that the row sums of $\LaN$ equal $1$ (meaning that $\LaN$ is stochastic) is equivalent to the simplest particular case of \eqref{eq5.F} corresponding to $\nu=\varnothing$,  and the uniqueness of the matrix  is proved easily (Lemma \ref{lemma5.uni} below).  

As in the context of section \ref{sect4}, we derive the desired coherency relation \eqref{eq5.F} from Okounkov's formula \eqref{eq3.B}. However, we can no longer simply substitute $Z=X$ into \eqref{eq3.B}. The reason is that in doing so, for general $(q,t)$ and $N\ge3$, we can stumble upon singularities. This is shown in Example \ref{example5.A}. 

Our strategy is the following. First, we represent the multiple $q$-integral on the left-hand side of \eqref{eq3.B} as a multiple sum (Lemma \ref{lemma5.A} below). Next, instead of directly substitute $Z=X$ into this sum, we specialize $z_i\to x_i$ step by step, by sorting out the coordinates in a special order depending on $k=k(X)$, as indicated in Definition \ref{def5.C} below. (A somewhat similar trick works in a different context, related to the $R$-matrix formalism, see e. g. Nazarov--Tarasov \cite[\S2]{NT}.)  

We show that, under assumption that $t$ is in general position with respect to $q$, in the course of this procedure, each term of our multiple sum remains well defined and finally has a limit. Many terms actually vanish in the limit, but those that survive give a weighted sum over the configurations $Y\prec X$, just as on the left-hand side of \eqref{eq5.F}. 

After that we check that in the final sum, the constraint on $t$ becomes inessential and can be dropped. Moreover, the weights $\LaN(X,Y)$ turn out to be strictly positive.  

\subsection{Uniqueness claim in Theorem A}\label{sect5.3}
This claim follows from the next lemma, which is similar to \cite[Lemma 4.1]{Ols-2016}. We fix $K\in\{1,2,\dots\}$ and define 
$\wt\Om_K$ as the disjoint union of the sets $\Om_k$ with $0\le k\le K$. Next, given $A>0$, we denote by $\wt\Om_K[-A,A]$ the subset of those configurations in $\wt\Om_K$ that are entirely contained in the closed interval $[-A,A]$. Recall that $\Sym(K)$ is our notation for the algebra of symmetric polynomials in $K$ variables. Any polynomial $P\in\Sym(K)$ can be viewed as a function on $\wt\Om_K[-A,A]$, with zeroes added as arguments if needed. Given a finite measure $M$ on $\wt\Om_K[-A,A]$, we can form its pairing  $\langle M, P\rangle$ with any $P\in\Sym(K)$:
$$
\langle M, P\rangle:=\sum_{X\in\wt\Om_K[-A,A]} P(X)M(X).
$$
This makes sense because $\wt\Om_K[-A,A]$ (as well as $\wt\Om_K$) is a countable set; here we tacitly assume that the sigma-algebra on $\wt\Om_K$ is generated by the singletons. 

\begin{lemma}\label{lemma5.uni}
In this notation,  any finite measure on $\wt\Om_K[-A,A]$ is uniquely determined by its pairings with the  Macdonald polynomials indexed by partitions $\nu\in\Y(K)$. 
\end{lemma} 

\begin{proof}
In section \ref{sect6.1} we define a nontrivial Hausdorff topology on $\wt\Om_K$. In this topology,  $\wt\Om_K[-A,A]$ is a compact subset. Although the topology is not discrete, this does not affect the Borel structure of the countable space $\wt\Om_K$:  the Borel sets  are simply arbitrary subsets, as for the discrete topology. Therefore, any finite measure on $\wt\Om_K[-A,A]$ is uniquely determined by its pairings with continuous functions.

Further, polynomials $P\in\Sym(K)$ are continuous functions on $\wt\Om_K[-A,A]$. Moreover, they separate points. Therefore, by virtue of the Stone--Weierstrass theorem, they form a dense subalgebra in the algebra of real-valued continuous functions on $\wt\Om_K[-A,A]$. Finally, the Macdonald polynomials $P_{\nu\mid K}(\ccdot;q,t)$ with $\nu\in\Y(K)$ form a basis in $\Sym(K)$. This proves the lemma. 
\end{proof}

In the context of Theorem A, we apply the lemma to the measure $\LaN(X, \ccdot)$ by taking $K=N-1$ and choosing the interval $[-A,A]$ so large that it contains all points of $X$. This will guarantee that the measure $\La^N_{N-1}(X,\ccdot)$ is supported on 
$\wt\Om_{N-1}[-A,A]$ because of the interlacement condition.

\subsection{Reorganization of Okounkov's formula}
To write the $q$-integral \eqref{eq3.B} as a multiple sum we need to introduce a suitable notation.  

Let 
$$
\eps:=(\eps(1),\dots,\eps(N-1))\in\{0,1\}^{N-1}
$$
denote an arbitrary binary word of length $N-1$, and let
$$
m:=(m_1,\dots,m_{N-1})\in\Z_{\ge0}^{N-1}
$$
be an arbitrary $(N-1)$-tuple of nonnegative integers. Given $\eps$ and $m$ as above,  we assign to every ordered $N$-tuple  $Z=(z_1,\dots,z_N)\in\C^N$ an $(N-1)$-tuple $\wt Z\in\C^{N-1}$:
$$
\wt Z:=(\wt z_1,\dots, \wt z_{N-1}):=(z_{1+\eps(1)}q^{m_1},\dots, z_{N-1+\eps(N-1)}q^{m_{N-1}}).
$$

Next, we set 
\begin{equation}\label{eq5.R_epsm}
 R_{\eps,m}(Z;q,t):=\prod_{r=1}^{N-1}\prod_{s=1}^N\frac{(\wt z_rq/z_s;q)_\infty}{(\wt z_rt/z_s;q)_\infty} \prod_{1\le i\ne j\le N}\frac{(z_it/z_j;q)_\infty}{(z_iq/z_j;q)_\infty}.
\end{equation}
Or, in more detailed notation,
\begin{equation}\label{eq5.E}
 R_{\eps,m}(Z;q,t):=\prod_{r=1}^{N-1}\prod_{s=1}^N\frac{(z_{r+\eps(r)}q^{m_r+1}/z_s;q)_\infty}{(z_{r+\eps(r)} q^{m_r} t/z_s;q)_\infty} \prod_{1\le i\ne j\le N}\frac{(z_it/z_j;q)_\infty}{(z_iq/z_j;q)_\infty}.
\end{equation}
Note that $(a;q)_\infty$ is an entire function in $a\in\C$, with simple zeroes at the points $1,q^{-1}, q^{-2},\dots$\,. If follows that $ R_{\eps,m}(Z;q,t)$ is a meromorphic function in $N$ variables $z_1,\dots,z_N$. Because this function actually depends only on the ratios of the variables, it can be regarded as a meromorphic function on the projective space $\C\mathbb P^{N-1}$.

Finally, we set
\begin{equation}\label{eq5.K}
\wt C_N(q,t):=(1-q)^{N-1} C_N(q,t)=\frac{((t;q)_\infty)^N}{(t^N;q)_\infty((q;q)_\infty)^{N-1}}
\end{equation}

\begin{lemma}\label{lemma5.A}
The left-hand side of Okounkov's formula \eqref{eq3.B}  can be written as the following multiple series
\begin{equation*}
\frac{\wt C_N(q,t)}{V(Z)}\sum_{\eps\in\{0,1\}^{N-1}}\sum_{m\in\Z_{\ge0}^{N-1}}V(\wt Z)  R_{\eps,m}(Z;q,t)  \prod\limits_{r=1}^{N-1} (-1)^{\eps(r)}\wt z_r\ F(\wt Z) ,
\end{equation*}
where
$$
F(\wt Z):=\frac{P_{\nu\mid N-1}(\wt z_1,\dots,\wt z_{N-1};q,t)}{(t^{N-1};q,t)_\nu}.
$$
\end{lemma}

\begin{proof}
This follows from the very definition of the $q$-integral. 
Indeed, we represent each of the one-variate $q$-integrals in \eqref{eq3.B} as the difference of two $q$-integrals,
\begin{equation}\label{eq5.D}
\int^{z_r}_{z_{r+1}} (\ccdot)d_q w_r=\int^{z_r}_0 (\ccdot)d_q w_r-\int^{z_{r+1}}_0(\ccdot)d_q w_r,
\end{equation}
and then write the $q$-integrals on the right as infinite series over $m_r\in\Z_{\ge0}$. Namely, in the first $q$-integral we set $w_r=z_r q^{m_r}$, and in the second $q$-integral we set $w_r=z_{r+1} q^{m_r}$. Then the whole $(N-1)$-fold $q$-integral in \eqref{eq3.B} turns into the sum of $2^{N-1}$ summands each of which is a series over  $m\in\Z_{\ge0}^{N-1}$. The summands are indexed by the binary words $\eps$: namely,  $\eps(r)=0$ encodes the choice of the first summand on the right-hand side of \eqref{eq5.D}, while $\eps(r)=1$ encodes the choice of the second one.
\end{proof}

\subsection{Singularities of the function $ R_{\eps,m}(Z;q,t)$} 
Let  us fix arbitrary $\eps\in\{0,1\}^{N-1}$, $m\in\Z_{\ge0}^{N-1}$, and $X=(x_1>\dots>x_N)\in\Om_N$. Next, suppose that $t$ is in general position with respect to $q$: in fact it suffices to require that 
\begin{equation}\label{eq5.G}
t,t^2,\dots,t^{N-1}\notin q^\Z:=\{q^n: n\in\Z\}.
\end{equation}

\begin{lemma}\label{lemma5.B}
Under these assumptions, the only factors in the denominators of \eqref{eq5.E} that  may vanish at the point $Z=X$ are those of the form $(z_{r+\eps(r)}q^{m_r}t/z_s;q)_\infty$, where one the following two conditions holds

{\rm(1)} $x_{r+\eps(r)}>0$, $x_s>0$, and $s=r+\eps(r)+1$,

{\rm(2)} $x_{r+\eps(r)}<0$, $x_s<0$, and $s=r+\eps(r)-1$.

\end{lemma}

\begin{proof}
Let us begin with the second product in \eqref{eq5.E}. The factors in the denominator have the form $(z_i q/z_j;q)_\infty$ with $i\ne j$. We claim that they do not vanish at $Z=X$. Indeed, if $x_i$ and $x_j$ are of opposite sign, then $x_i q/x_j$ is negative and hence $(x_i q/x_j;q)_\infty>0$. If $x_i$ and $x_j$ are of same sign, then $x_i q/x_j$ lies in $t^{i-j}q^\Z$.  Since $i\ne j$, we see from \eqref{eq5.G} that $x_i q/x_j$ is not in $q^\Z$, which entails  $(x_i q/x_j;q)_\infty\ne0$. 

Let us turn to the first product in \eqref{eq5.E}. The factors in the denominator, at $Z=X$, have the form 
$$
(x_{r+\eps(r)}q^{m_r} t/x_s;q)_\infty, \qquad r=1,\dots,N-1, \quad s=1,\dots,N, \quad m_r\in\Z_{\ge0}.
$$
Again, if $x_{r+\eps(r)}$ and $x_s$ are of opposite sign, then vanishing is impossible. Assume they are of the same sign. Then vanishing may happen only if $x_{r+\eps(r)} t/x_s\in q^\Z$, which exactly means that either (1) or (2) holds. 
\end{proof}

Whether a singularity really occurs depends also on the factors in the numerator. At first glance, the whole picture looks complicated, but our limit procedure makes it possible to avoid singularities.

\begin{definition}\label{def5.A}
Given $X\in\Om_N$, let $k=k(X)$ be the number of positive coordinates in $X$, so that  
$$
x_1>\dots>x_k>0>x_{k+1}>\dots>x_N.
$$
We say that a given binary word $\eps\in\{0,1\}^{N-1}$ is \emph{$k$-adapted} if
$$
\eps(r)=0 \quad \text{for} \quad 1\le r\le k-1 \quad \text{and} \quad \eps(r)=1 \quad \text{for} \quad k+1\le r\le N-1.
$$ 

Note that if $k=N$ (that is, all $x_i$'s are positive) or  $k=0$ (that is, all $x_i$'s are negative), then there is a unique choice for $\eps$.  In the remaining cases, when $1\le k\le N-1$, there are two $k$-adapted $\eps$'s,  because $\eps(k)$ may take both values, $0$ and $1$. 
\end{definition}

\begin{corollary}\label{cor5.A}
Suppose, as above, that $t$ satisfies the constraint \eqref{eq5.G}. Let $X\in\Om_N$ and assume that $\eps$ is $k(X)$-adapted. 

Then the meromorphic function $ R_{\eps,m}(Z;q,t)$ is nonsingular at $Z=X$ for any $m\in\Z_{\ge0}^{N-1}$. 
\end{corollary}

\begin{proof} 
Set $k=k(X)$. We split $R_{\eps,m}(Z;q,t)$ into the product of two expressions, 
$$
 R_{\eps,m}(Z;q,t)= R^{(1)}_{\eps,m}(Z;q,t) R^{(2)}_{\eps,m}(Z;q,t),
$$
where
\begin{equation*}
 R^{(1)}_{\eps,m}(Z;q,t):=\prod_{r=1}^{k-1}\frac{(z_r t/z_{r+1};q)_\infty}{(z_r q^{m_r}t/z_{r+1};q)_\infty} \cdot \prod_{r=k+1}^{N-1}\frac{(z_{r+1} t/z_r;q)_\infty}{(z_{r+1} q^{m_r}t/z_r;q)_\infty}
\end{equation*}
and $ R^{(2)}_{\eps,m}(Z;q,t)$ consists of the remaining factors. By virtue of Lemma \ref{lemma5.B}, all factors in the denominator of $R_{\eps,m}(Z;q,t)$  that may vanish at $Z=X$ are assembled in the denominator of $R^{(1)}_{\eps,m}(Z;q,t)$, so that $ R^{(2)}_{\eps,m}(Z;q,t)$ is nonsingular at $X$. On the other hand, $ R^{(1)}_{\eps,m}(Z;q,t)$ can be written as
\begin{equation}
 R^{(1)}_{\eps,m}(Z;q,t)=\prod_{r=1}^{k-1}(z_r t/z_{r+1};q)_{m_r} \cdot \prod_{r=k+1}^{N-1}(z_{r+1} t/z_r;q)_{m_r}
\end{equation}
and hence is nonsingular, too.
\end{proof}

\subsection{Example of singularity}

One can show that for $N=2$, the function $ R_{\eps,m}(Z;q,t)= R_{\eps,m}(z_1,z_2;q,t)$ is nonsingular at each point $X\in\Om_2$, for any $(\eps,m)$. The next example shows that this is not always true for $N\ge3$.  

\begin{example}\label{example5.A}
Below we use the standard shorthand notation
$$
(a_1,\dots,a_n;q)_\infty:=\prod_{i=1}^n (a_i;q)_\infty.
$$
Let $N=3$, $\eps=(1,0)$, and $m=(0,0)$. The corresponding function $ R_{\eps,m}(z_1,z_2,z_3;q,t)$ can be represented in the form
\begin{equation*}
\frac{(z_1t/z_3, z_3t/z_1, z_3t/z_2, z_2q/z_3, z_2q/z_1,z_1tq/z_2, q,q;q)_\infty\, z_3}
{(z_1q/z_2, z_1q/z_3, z_3q/z_1, z_3q/z_2,z_2t/z_1, z_2tq/z_3, t,t;q)_\infty\,z_2}
\cdot \frac{z_2-z_1t}{z_3-z_2t}.
\end{equation*}
Suppose that $t,t^2\notin q^\Z$. At the point $X=(1,t,t^2)$ (which is contained in $\Om_3$),  the first fraction is nonsingular and nonvanishing, while the second fraction has a singularity. Hence the whole expression is also singular at $X$.  
\end{example}

\subsection{Admissible pairs $(\eps,m)$}\label{sect5.adm}

Let $X\in\Om_N$ and $k:=k(X)$ (Definition \ref{def5.A}). We use the description of  $\Om_N$ given in Lemma \ref{lemma5.Omega}. As pointed out in Lemma \ref{lemma5.interlace}, there exist nonnegative  integers $l_1,\dots,l_{N-1}$ such that
\begin{equation*}
x_{i+1}=x_i t q^{l_i}, \quad i=1,\dots,k-1; \qquad x_i=x_{i+1} t q^{l_i}, \quad i=k+1,\dots,N-1.
\end{equation*}

\begin{definition}\label{def5.B}
Let $\eps\in\{0,1\}^{N-1}$ and $m\in\Z_{\ge0}^{N-1}$. Let us say that a pair $(\eps,m)$ is \emph{$X$-admissible} if $\eps$ is $k$-adapted in the sense of Definition \ref{def5.A} and, moreover, $m_i\le l_i$ for all $i=1,\dots,N-1$, except $i=k$.
\end{definition}
 
In the next lemma we restate the interlacement relation $Y\prec X$. 

\begin{lemma}\label{lemma5.F}
There is a one-to-one correspondence between the $X$-admissible pairs $(\eps,m)$ and  the configurations $Y=(y_1>\dots>y_{N-1})\prec X$, defined by
$$
y_i=x_{i+\eps(i)}q^{m_i}, \quad i=1,\dots,N-1.
$$
That is,
$$
y_i=x_i q^{m_i}, \quad i=1,\dots,k-1; \qquad y_i=x_{i+1}q^{m_i}, \quad i=k+1,\dots,N-1,
$$
and $y_k=x_k q^{m_k}$ or $y_k= x_{k+1}q^{m_k}$ depending on whether $\eps(k)$ equals $0$ or $1$. 
\end{lemma}

\begin{proof}
This follows at once from Lemma \ref{lemma5.interlace}.
\end{proof}

\subsection{Limit transition in $R_{\eps;m}(Z;q,t)$ as $Z\to X$ ($t$ generic)}
Fix an arbitrary $X\in\Om_N$, where $N\ge 3$ and denote $k:=k(X)$. We would like to pass to a limit in \eqref{eq5.E} as $Z\to X$. Example \ref{example5.A} shows that one cannot do this directly, because, for some $(\eps,m)$, the meromorphic function $ R_{\eps,m}(Z;q,t)$ may be singular at $Z=X$. To circumvent this obstacle, we let $Z$ approach $X$ along a special path, as described in the following definition.

\begin{definition}[Step-by-step limit transition]\label{def5.C}

(1) If $k=N$, then we first set $z_1=x_1$ keeping $z_2,\dots,z_N$ in general position, next we let $z_2\to x_2$, then we let $z_3\to x_3$, and so on up to the step $z_N\to x_N$. 

(2) If $k=0$, then we perform the same procedure in the inverse direction: $z_N=x_N$, $z_{N-1}\to x_{N-1}$, \dots, $z_1\to x_1$.

(3) If $1\le k\le N-1$, then we combine the two procedures, that is, we begin with $z_1=x_1$ and go up to the step $z_k\to x_k$; then we set $z_N=x_N$ and go up to the step $z_{k+1}\to x_{k+1}$. Equally well we may stop at the step  $z_{k-1}\to x_{k-1}$, then pass to $z_N=x_N$ and proceed up to the step $z_{k+2}\to x_{k+2}$, and finally set $z_k=x_k$, $z_{k+1}=x_{k+1}$. 

\end{definition}

\begin{lemma}\label{lemma5.D}
Assume $t$ satisfies the constraint \eqref{eq5.G}. Fix an arbitrary $X\in\Om_N$ and let $Z$ approach $X$ as described in Definition \ref{def5.C}.  Then $ R_{\eps,m}(X;q,t)$ has a limit for every $(\eps,m)$, and the limit value equals\/ $0$ unless $(\eps,m)$ is $X$-admissible in the sense of Definition \ref{def5.B}. 
\end{lemma}

\begin{proof}

Suppose that $k\ge2$. As was pointed out above, the function $Z\mapsto  R_{\eps,m}(Z;q,t)$ depends only on the ratios $z_i/z_j$ of the variables. It follows that the first step, the specialization $z_1=x_1$, presents no difficulty, because all the ratios $z_i/z_j$ remain in general position. Let us justify the second step, the limit transition $z_2\to x_2$. 

Setting
\begin{gather*}
X':=(x_2,\dots,x_N), \quad Z':=(z_2,\dots,z_N), \\
\eps':=(\eps(2),\dots,\eps(N-1)), \quad m':=(m_2,\dots,m_{N-1}),
\end{gather*}
we may write
$$
 R_{\eps,m}(Z;q,t)= R_{\eps',m'}(Z';q,t) \wt R_{\eps(1),m_1}(Z;q,t),
$$
where $\wt R_{\eps(1),m_1}(Z;q,t)$ collects all factors from  $ R_{\eps,m}(Z;q,t)$ that are not contained in $ R_{\eps',m'}(Z';q,t)$. The exact form of $\wt R_{\eps(1),m_1}(Z;q,t)$ depends on whether $\eps(1)$ equals $0$ or $1$; let us examine these two possible variants separately.  

$\bullet$ \emph{Variant} 1: $\eps(1)=0$. Then we have
\begin{multline}\label{eq5.H}
\wt R_{0,m_1}(Z;q,t)\big|_{z_1=x_1}=\prod_{s=1}^N\frac{(z_1 q^{m_1+1}/z_s;q)_\infty}{(z_1 q^{m_1}t/z_s;q)_\infty}\cdot \prod_{j=2}^N\frac{(z_1 t/z_j, \, z_j t/x_1; q)_\infty}{(x_1 q/z_j,\, z_j q/z_1; q)_\infty}\bigg|_{z_1=x_1}\\
=\frac{(q^{m_1+1};q)_\infty}{(q^{m_1}t;q)_\infty}\cdot  \prod_{s=2}^N \frac{(z_1 t/z_s;q)_{m_1}}{(x_1 q/z_s; q)_{m_1}}\cdot \prod_{j=2}^N \frac{(z_j t/x_1; q)_\infty}{(z_j q/x_1; q)_\infty}.
\end{multline}
Look at the final expression in \eqref{eq5.H}.  Here the first fraction is a strictly positive constant. The product over $j$ causes no problem: it can be directly specialized at $Z=X$, because then all its factors will be of the form $(a;q)_\infty$ with $a<1$ and hence are strictly positive. Let us turn now to the product over $s$ and write it as 
\begin{equation}\label{eq5.I}
\frac{(x_1 t/z_2)_{m_1}}{(x_1 q/z_2)_{m_1}}\cdot\prod_{s=3}^N \frac{(x_1 t/z_s)_{m_1}}{(x_1 q/z_s)_{m_1}}
\end{equation}
The product over $s$ does not involve $z_2$, so that it is not sensitive to the specialization $z_2\to x_2$.  Finally, examine the fraction in front of the product in \ref{eq5.I}. At the point $z_2=x_2=x_1q^{l_1}t$, the denominator does not vanish, because $t\notin q^\Z$. As for the numerator,  at the same point, it reduces to $(q^{-l_1};q)_{m_1}$, and this quantity vanishes unless $m_1\le l_1$. 

$\bullet$ \emph{Variant} 2: $\eps(1)=1$.  Then we have
\begin{equation}\label{eq5.J}
\wt R_{1,m_1}(Z;q,t)=\prod_{s=1}^N\frac{(z_2 q^{m_1+1}/z_s;q)_\infty}{(z_2 q^{m_1}t/z_s;q)_\infty}\cdot \prod_{j=2}^N\frac{(z_1 t/z_j, \, z_j t/z_1; q)_\infty}{(z_1 q/z_j,\, z_j q/z_1; q)_\infty}.
\end{equation}
This expression makes sense under the substitution $(z_1,z_2)=(x_1,x_2)$: indeed,  we use the fact that $(z_3,\dots,z_N)$ is  in general position, $t$ is also generic, and $x_2/x_1<1$. But at the point $(z_1,z_2)=(x_1,x_2)$, the factor $(z_1 t/z_2;q)_\infty$ turns into $(q^{-l_1};q)_\infty$ and hence vanishes. So the whole expression in fact disappears.

We conclude that the limit as $z_2\to x_2$ does exist, and only the terms with $\eps(1)=0$ and $m_1\le l_1$ survive. 

Then we apply the same procedure to the function $ R_{\eps',m'}(Z';q,t)$ and so on. If $k=N$, then we may go up to the end and obtain the desired result. Otherwise we stop as $k$ becomes equal $1$. 

After that we begin to move in the opposite direction, starting from $z_N=x_N$. (Or we do that from the very beginning if $k=0$.) Here the argument is similar.  

Finally, the only remaining case is the one with $N=2$, $k=1$. That is, $x_2<0<x_1$. Here the argument is trivial: in the limit, all $q$-Pochhammer factors are of the form $(a;q)_\infty$, where either $0<a<1$ or $a<0$; hence, they are strictly positive (see formulas \eqref{eq5.N=2(1)} and \eqref{eq5.N=2(2)} below).  In this case, all the summands survive in the limit and lead to the desired (infinite) sum over $\{y\}\prec X$.  \end{proof}

\subsection{Computation of $R_{\eps;m}(X;q,t)$ for arbitrary $t$ and admissible $(\eps;m)$}\label{sect5.9}

Let, as above, $X\in\Om_N$. The next lemma is a refinement of Corollary \ref{cor5.A}.

\begin{lemma}\label{lemma5.E}
Let us remove the constraints \eqref{eq5.G}, so that  $t\in(0,1)$ may be arbitrary.
If a pair $(\eps,m)$ is $X$-admissible in the sense of Definition \ref{def5.B}, then the function $ R_{\eps,m}(Z;q,t)$ is nonsingular at the point $Z=X$ and $ R_{\eps,m}(X;q,t)>0$.
\end{lemma}

\begin{proof}
Suppose $k\ge2$, so that $0<x_2<x_1$. Then from the proof of  Lemma \ref{lemma5.D} (formula  \eqref{eq5.H} and the argument after it) we obtain the recurrence relation  
\begin{multline}\label{eq5.recurr1}
R_{\eps_1,\dots,\eps_{N-1}; m_1,\dots,m_N}(x_1,\dots,x_N;q,t)\\
 =\frac{(q^{m_1+1};q)_\infty}{(q^{m_1}t;q)_\infty}\cdot  \prod_{s=2}^N \frac{(x_1 t/x_s;q)_{m_1}}{(x_1 q/x_s; q)_{m_1}}\cdot \prod_{j=2}^N \frac{(x_j t/x_1; q)_\infty}{(x_j q/x_1; q)_\infty}\\
 \times R_{\eps_2,\dots,\eps_{N-1}; m_2,\dots,m_N}(x_2,\dots,x_N;q,t),
\end{multline}
with the understanding that the expression on the last line equals $1$ if $N=2$. 

We claim that the expression on the middle line is strictly positive, for any $t\in(0,1)$. Indeed, in the product over $j$, all factors are strictly positive. Next, write the product over $s$ in a more detailed way:
$$
\prod_{s=2}^N \frac{(x_1 t/x_s;q)_{m_1}}{(x_1 q/x_s; q)_{m_1}}=\prod_{s=2}^N \frac{(1-x_1 t/x_s)\dots(1-x_1tq^{m_1-1}/x_s)}{(1-x_1 q/x_s)\dots(1-x_1q^{m_1}/x_s)}.
$$
Since $x_s\le x_1tq^{l_1}$ and $m_1\le l_1$, we see that all factors are strictly negative, so that the whole expression is strictly positive. 

Next, suppose $N-k\ge2$, so that $x_N<x_{N-1}<0$. Then we obtain a similar recurrence relation,
\begin{multline}\label{eq5.recurr2}
R_{\eps_1,\dots,\eps_{N-1}; m_1,\dots,m_N}(x_1,\dots,x_N;q,t)\\
 =\frac{(q^{m_N+1};q)_\infty}{(q^{m_N}t;q)_\infty}\cdot  \prod_{s=1}^{N-1} \frac{(x_N t/x_s;q)_{m_N}}{(x_N q/x_s; q)_{m_N}}\cdot \prod_{j=1}^{N-1} \frac{(x_j t/x_N; q)_\infty}{(x_j q/x_N; q)_\infty}\\
 \times R_{\eps_1,\dots,\eps_{N-2}; m_1,\dots,m_{N-1}}(x_1,\dots,x_{N-1};q,t).
\end{multline}
The same argument shows that the expression on the middle line is strictly positive. 

Using these two recurrence relations we reduce the problem to the case when $N=2$ and $k=1$, meaning $x_2<0<x_1$.  Then we have 
\begin{equation}\label{eq5.N=2(1)}
R_{\eps_1; m_1}(x_1,x_2;q,t)=\dfrac{(q^{m_1+1};q)_\infty}{(q^{m_1}t;q)_\infty}\dfrac{(x_1t/x_2;q)_{m_1}(x_2t/x_1;q)_\infty}{(x_1q/x_2;q)_{m_1}(x_2q/x_1;q)_\infty}, \qquad \eps(1)=0,
\end{equation}
and 
\begin{equation}\label{eq5.N=2(2)}
R_{\eps_1; m_1}(x_1,x_2;q,t)=\dfrac{(q^{m_1+1};q)_\infty}{(q^{m_1}t;q)_\infty}\dfrac{(x_2t/x_1;q)_{m_1}(x_1t/x_2;q)_\infty}{(x_2q/x_1;q)_{m_1}(x_1q/x_2;q)_\infty}, \qquad \eps(1)=1.
\end{equation}
In both variants, all the factors are strictly positive for any $t\in(0,1)$, because $x_2<0<x_1$. 
\end{proof}

\subsection{The matrices $\LaN$}\label{sect5.10}

\begin{definition}\label{def5.D}
For each $N=2,3,\dots$ we define a matrix $\LaN$ of format $\Om_N\times\Om_{N-1}$ as follows. 

$\bullet$ If $Y$ and $X$ do not interlace, then $\LaN(X,Y):=0$. 

$\bullet$. If $Y\prec X$, then we take the $X$-admissible pair $(\eps,m)$ corresponding to $Y$ (see Lemma \ref{lemma5.F}) and set
\begin{equation}\label{eq5.L}
\LaN(X,Y):=\frac{((t;q)_\infty)^N}{(t^N;q)_\infty((q;q)_\infty)^{N-1}}\cdot\frac{V(Y)}{V(X)}\prod_{i=1}^{N-1} |y_i|\cdot  R_{\eps,m}(X;q,t).
\end{equation}
\end{definition}

\smallskip

The definition makes sense, because we know from Lemma \ref{lemma5.E} that the meromorphic function $R_{\eps,m}(Z;q,t)$ is nonsingular at $Z=X$. Note also that the first fraction on the right-hand side is the constant $\wt C_N(q,t)$ given by \eqref{eq5.K}. 

To make the definition  \eqref{eq5.L} explicit we have to exhibit an explicit expression for $R_{\eps,m}(X;q,t)$. Here are a few ways to do that. 

\smallskip

(1) From the proof of Lemma \ref{lemma5.E} one can deduce the formula 
\begin{multline}\label{eq5.R}
R_{\eps;m}(X;q,t)=\prod_{r=1}^{N-1}\frac{(q^{m_r+1};q)_\infty}{(q^{m_r}t;q)_\infty} 
\cdot\prod_{r=1}^{N-1} \prod_{\substack{s=1,\dots,N\\ s\ne r+\eps(r)}}\frac{(x_{r+\eps(r)}t/x_s;q)_{m_r}}{(x_{r+\eps(r)}q/x_s;q)_{m_r}}\\
\times \prod_{\substack{j=1,\dots,N\\ j\ne k+1-\eps(k)}}\frac{(x_{k+1-\eps(k)}t/x_j;q)_\infty}{(x_{k+1-\eps(k))}q/x_j;q)_\infty}.
\end{multline}
Recall that $\eps(r)=0$ for $1<r<k$, $\eps(r)=1$ for $k<r<N$, while for $r=k$ with $0< k<N$, both values $0$ and $1$ are admitted. Note also that the last product in \eqref{eq5.R} should be removed when $k=0,N$. 

This formula can be checked directly as follows. It suffices to show that it agrees with the recurrence relations \eqref{eq5.recurr1}, \eqref{eq5.recurr2}, and with formulas \eqref{eq5.N=2(1)}, \eqref{eq5.N=2(2)}.

When we split off $x_1$, the remaining variables are renamed and $k$ is replaced by $k-1$. From this it is seen that \eqref{eq5.R} agrees with \eqref{eq5.recurr1}.

When we split off $x_N$, the enumeration does not change and $k$ remains intact. This agrees with \eqref{eq5.recurr2}.

Examine now the case when $N=2$ and $k=1$, meaning that $x_2<0<x_1$. There are two variants, $\eps(1)=0$ and $\eps(1)=1$. In both variants, agreement with \eqref{eq5.N=2(1)} and \eqref{eq5.N=2(2)} is seen directly. 

(2)  An alternative formula is obtained by specializing directly $Z:=X$, $\wt Z:=Y$ into \eqref{eq5.R_epsm}:
\begin{equation}\label{eq5.LaN}
R_{\eps;m}(X;q,t)= \prod_{r=1}^{N-1}\prod_{s=1}^N\frac{(y_rq/x_s;q)_\infty}{(y_r t/x_s;q)_\infty} \prod_{1\le i\ne j\le N}\frac{(x_it/x_j;q)_\infty}{(x_iq/x_j;q)_\infty}
\end{equation}
However, in this formula we need to impose the constraint \eqref{eq5.G} to guarantee that the factors in the denominators do not vanish. 

\smallskip

(3) Next, one can get rid of the constraint \eqref{eq5.G} in \eqref{eq5.LaN} by the following transformation of the right-hand side. The problem consists in possible vanishing of the factors $(y_r t/x_s;q)_\infty$ (where $y_r$ and $x_s$ are of the same sign and such that $|y_rt|\ge|x_s|$), as well as of the factors $(x_iq/x_j;q)_\infty$ (where $x_i$ and $x_j$ are of the same sign and such that $|x_iq|\ge|x_j|$). In \eqref{eq5.R}, that problem was resolved due to cancellations. Here is the idea of another solution.

Let, for definiteness, $1\le r<s<k$, so that $x_r\ge y_r>x_s>0$. Using the fact that $y_r\in x_r q^\Z$, one can show that 
\begin{equation}\label{eq5.psi}
\frac{(y_rq/x_s;q)_\infty(x_rt/x_s;q)_\infty}{(y_rt/x_s;q)_\infty(x_rq/x_s;q)_\infty}
=\left(\frac{y_r}{x_r}\right)^{\tau-1} 
\dfrac{(x_sq/(y_rt);q)_\infty (x_s/x_r;q)_\infty}{(x_s/y_r;q)_\infty(x_sq/(x_rt);q)_\infty}.
\end{equation}
Observe that on the right-hand side of \eqref{eq5.psi}, the factors in the denominator already do not vanish.

This trick makes it possible to transform \eqref{eq5.LaN} to the form which does not require the constraint \eqref{eq5.G}.

\subsection{Completion of proof of Theorem A}

As explained in section \ref{sect5.2}, for the proof of Theorem A it suffices to prove the identity \eqref{eq5.F} linking the matrices $\LaN$ with Macdonald polynomials. 

Suppose first that $t$ satisfies the constraint \eqref{eq5.G} 
By virtue of Lemma \ref{lemma5.A}, Okounkov's formula \eqref{eq3.B} can be written in the form
\begin{multline*}
\frac{\wt C_N(q,t)}{V(Z)}\sum_{\eps\in\{0,1\}^{N-1}}\sum_{m\in\Z_{\ge0}^{N-1}}V(\wt Z)  R_{\eps,m}(Z;q,t)  \prod\limits_{i=1}^{N-1} (-1)^{\eps(i)}\wt z_i\\ 
\times \frac{P_{\nu\mid N-1}(\wt z_1,\dots,\wt z_{N-1})}{(t^{N-1};q,t)_\nu}
=\frac{P_{\nu\mid N}(z_1,\dots,z_{N-1})}{(t^N;q,t)_\nu}
\end{multline*}
Now we fix an arbitrary $X\in\Om_N$ and let $Z$ approach $X$ in the way described in Definition \ref{def5.C}. By Lemma \ref{lemma5.D}, in this limit regime, $ R_{\eps,m}(Z;q,t)$ has a limit for each $(\eps,m)$, but the result vanishes unless $(\eps,m)$ is $X$-admissible. Therefore, in the limit,  we obtain on the left a sum over the $X$-admissible pairs $(\eps,m)$, which can be interpreted as a sum over the configurations $Y\prec X$. Note that, given an $X$-admissible pair $(\eps, m)$, we have $\wt Z\to Y$ and $(-1)^{\eps(i)}$ is the sign of $y_i$, so that $(-1)^{\eps(i)}\wt z_i\to |y_i|$. Taking into account Definition \ref{def5.D} we see that in the limit, Okounkov's formula turns into \eqref{eq5.F}, as desired. 

After that we may remove the constraint on $t$ by continuity, because the resulting formula makes sense for any $t\in(0,1)$ due to the results of section \ref{sect5.9}.

This completes the proof of Theorem A.  

\subsection{Remarks}

\subsubsection{The Dixon--Anderson kernel} 

Let $\Conf_N(\R)$ denote the set of $N$-point configurations on $\R$: an element of $\Conf_N(\R)$ is an $N$-tuple $\XX=(x_1>\dots>x_N)$ of real numbers. If $\XX\in\Conf_N(\R)$ and $\YY\in\Conf_{N-1}(\R)$, then we write $\XX\succ \YY$ or $\YY\prec\XX$ if $x_i>y_i>x_{i+1}$ for $i=1,\dots,N-1$. 

Given $\XX\in\Conf_N(\R)$ with $N\ge2$, the following formula defines a probability measure $\LL^N_{N-1}(\XX, d\YY)$ on the domain $\{\YY: \YY\prec\XX\}\subset\Conf_{N-1}(\R)$:
\begin{equation}\label{eq5.DA}
\LL^N_{N-1}(\XX, d\YY)=\frac{\Ga(N\tau)}{(\Ga(\tau))^N} \, \frac{V(\YY)}{(V(\XX))^{2\tau-1}} \, (V(\XX;\YY))^{\tau-1} d\YY,
\end{equation}
where $\tau>0$ is a parameter and 
$$
V(\XX):=\prod_{1\le i<j\le N}(x_i-x_j), \quad V(\YY):=\prod_{1\le r<s\le N-1}(y_r-y_s), \quad V(\XX;\YY):=\prod_{i=1}^N \prod_{r=1}^{N-1}|x_i-y_r|.
$$

The fact that $\LL^N_{N-1}(X,\ccdot)$ is indeed a probability measure is equivalent to the evaluation of what is called a Dixon--Anderson integral (Forrester--Warnaar \cite[sect. 2.1]{FW-2008}).  Thus, $\LL^N_{N-1}$ is a Markov kernel from $\Conf_N(\R)$ to $\Conf_{N-1}(\R)$; let us call it the \emph{Dixon--Anderson kernel}. 

For more about it, see Assiotis--Najnudel \cite{AN}. 

\subsubsection{Degeneration $\LaN\to\LL^N_{N-1}$}

Fix $\tau>0$ and suppose $t=q^\tau$, as usual. One can show that, as $q\to1$, the stochastic matrices $\LaN$ (Definition \ref{def5.D}) converge to the Dixon--Anderson kernels $\LL^N_{N-1}$. 

In the case $\tau\in\Z_{\ge1}$ the proof is easy: one can use the simple formula \eqref{eq4.B}. For arbitrary $\tau>0$ one has to deal with the more sophisticated definition \eqref{eq5.L} and the expression \eqref{eq5.LaN} (or rather its transformation described in item (3) of subsection \ref{sect5.10}). Then the proof relies on the asymptotic formula 
$$
\lim_{q\to1}\frac{(u q^A;q)_\infty}{(u q^B;q)_\infty}=(1-u)^{B-A},
$$
which is valid on the domain $\C\setminus[1,+\infty)$ (Andrews--Askey--Roy \cite[Theorem 10.2.4]{AAR}, Gasper--Rahman \cite[ch. 1, (3.19)]{GR}). 

\subsubsection{Continuous analogue of Theorem A}

The following identity is a continuous analogue of the coherency relation \eqref{eq5.F}:
\begin{equation}\label{eq5.F1}
\int_{\YY\prec\XX}\LL^N_{N-1}(\XX,d\YY) \frac{P_{\nu\mid N-1}(\YY;\tau)}{(N-1;\tau)_\nu}=\frac{P_{\nu\mid N}(\XX;\tau)}{(N;\tau)_\nu},
\end{equation}
where $X\in\Conf_N(\R)$, $\nu\in\Y(N-1)$, 
$$
(u;\tau)_\nu:=\prod_{(i,j)\in\nu}((u+1-i)\tau+j-i),
$$
and the polynomials on the left and on the right are the Jack polynomials with parameter $\tau$, in $N-1$ and $N$ variables, respectively.  

This formula appeared in \cite[sect. 6]{OO-1997}. Note that it is different from \eqref{eq1.E}.

\section{Extended stochastic matrices $\wt\La^N_{N-1}$}\label{sect6}

This section serves as a preparation to the proof of Theorem B. In that proof we use a compactness argument, for which we need to deal with some larger sets of configurations $\wt\Om_N\supset\Om_N$ equipped with a non-discrete  topology. We introduce certain matrices $\wt\La^N_{N-1}$ of format $\wt\Om_N\times\wt\Om_{N-1}$ that extend the matrices $\LaN$, and we establish a technical result --- a property of $\wt\La^N_{N-1}$ stated as Theorem \ref{thm6.A}. It is used in the sequel, in sections \ref{sect7.3} and \ref{sect8.2}.

\subsection{Preliminaries}\label{sect6.1}

Theorem B was formulated in terms of double signatures, but it is more convenient to deal directly with point configurations, as in section \ref{sect5}. The definitions formulated below are a direct extension of those given in the author's paper \cite[\S6]{Ols-2016}. The reader is referred to that paper for more details.

\begin{definition}\label{def6.A}
Recall that an \emph{infinite signature} is an infinite sequence of non increasing integers $a=(a_1\ge a_2\ge\dots)$.

(i) By definition, the set $\Om_\infty$ consists of the configurations in $\R^*$ of the form
$$
X(a^+,a^-):=\{\zeta_+ q^{-a^+_i}t^{i-1}\}\cup\{\zeta_- q^{-a^-_j}t^{j-1}\},
$$
where $a^+$ and $a^-$ are two signatures, and at least one of them is infinite. 

(ii) Next, $\wt\Om$ is defined as the union of the sets $\Om_0:=\{\varnothing\}$, $\Om_1$, $\Om_2$, \dots, and $\Om_\infty$.
\end{definition}

Thus, elements of the space $\wt\Om$ are certain configurations in $\R^*$ which may be finite or infinite (in the case $t=q$ the space $\wt\Om$ coincides with the set $\wt{\mathbb G}_\infty$ from \cite{Ols-2016}).  We equip $\wt\Om$ with a structure of uniform space by proclaiming two configurations $X,X'\in\wt\Om$ to be \emph{$\eps$-close} (where $\eps>0$ is small)  if they coincide outside the interval $(-\eps,\eps)$. In particular, this makes $\wt\Om$ a topological space. As such, it is locally compact and metrizable. 

Both $\Om_\infty$ and $\bigcup_{N=0}^\infty\Om_N$ are dense subsets of $\wt\Om$. 

For each $N=1,2,\dots$, we denote by $\wt\Om_N$ the closure of $\Om_N$ in $\wt\Om$; it is the union of the sets $\Om_0,\dots,\Om_N$.

\subsection{Construction of matrices $\wt\La^N_{N-1}$}

Below the symbol $\P(\ccdot)$ denotes the space of probability Borel measures on a given topological space. 

\begin{proposition}\label{prop6.A}
Fix $N=2,3,\dots$ and let us interpret the matrix $\LaN$ as a map $\Om_N\to\P(\Om_{N-1})$. As such, it can be uniquely extended to a map $\wt{\La}^N_{N-1}: \wt\Om_N\to \P(\wt\Om_{N-1})$, which is continuous with respect to the topology on $\wt\Om_N$ inherited from $\wt\Om$ and the weak topology on $\P(\wt\Om_{N-1})$.  
\end{proposition}

\begin{proof}
The argument is exactly the same as in the proof of \cite[Proposition 4.2]{Ols-2016}. The key ingredient is the coherency relation \eqref{eq5.F} for Macdonald polynomials, which is a generalization of a similar relation for the Schur polynomials from \cite[Proposition 2.4]{Ols-2016}. 
\end{proof}

In the theorem below we use the fact that for each $N$, the $N$-variate Macdonald polynomials can be extended from $\Om_N$ to the ambient space $\wt\Om_N$ by continuity. This is equivalent to saying that the values on the subset $\Om_n\subset \wt\Om_N$ with $n<N$ are obtained by adding extra $N-n$ zeroes as variables.

\begin{theorem}\label{thm6.B}
Let $N=2,3,\dots$, $\nu\in\Y(N-1)$, and $X\in\wt\Om_N$. Then 
\begin{equation}\label{eq6.A}
\sum_{Y\in\wt\Om_{N-1}}\wt{\La}^N_{N-1}(X,Y)\frac{P_{\nu\mid N-1}(Y;q,t)}{(t^{N-1};q,t)_\nu}=\frac{P_{\nu\mid N}(X;q,t)}{(t^N;q,t)_\nu}.
\end{equation}
\end{theorem}

\begin{proof}
This follows from  \eqref{eq5.F}) and the definition of the extended matrices. 
\end{proof}

\subsection{The support of $\wt\La^N_{N-1}(X^*, \ccdot)$: statement of the result}

Let $n<N$. We are going to define a modified interlacement relation, denoted as $Y^*\prec\!\!\prec X^*$, between configurations $X^*\in\Om_n\subset\wt\Om_N$ and $Y^*\in\Om_n\subset\wt\Om_{N-1}$. Introduce a notation: 
\begin{equation}\label{eq6.B}
k(X^*):=\#\{x^*\in X^*: x^*>0\}, \quad l(X^*):=\#\{x^*\in X^*: x^*<0\}.
\end{equation} 
We will often abbreviate and write $k=k(X^*)$, $l=l(X^*)$. Evidently, $k+l=n$. 

\begin{definition}\label{def6.B}
We write $Y^*\prec\!\!\prec X^*$ if $Y^*\prec (X^*\cup\{x^0\})$, where $x^0$ is an arbitrary point in $\zeta_+ q^\Z t^k\sqcup \zeta_- q^{\Z}t^l$ sufficiently close to zero. Note that $X^*\cup\{x^0\}\in \wt\Om_N$ and the choice of $x^0$ does not matter (provided it is close to $0$). 
\end{definition}

Equivalently, writing $X^*=(x^*_1>\dots>x^*_n)$ and $Y^*=(y^*_1>\dots>y^*_n)$, the relation $Y^*\prec\!\!\prec X^*$ means that
\begin{equation*}
\begin{gathered}
x^*_{i+1}t^{-1}\le y^*_i\le x^*_i \quad \text{for} \quad 1\le i\le k-1,\\
x^*_j\le y^*_j\le x^*_{j-1}t^{-1}\quad \text{for} \quad  k+2\le j\le n,\\
0<y^*_k\le x^*_k,  \qquad x^*_{k+1}\le y^*_{k+1}<0,
\end{gathered}
\end{equation*}
with the understanding that if $k=0$ or $l=0$, then some of the above inequalities disappear. 

For instance, if $n=3$ and $X^*=(x^*_1,x^*_2,x^*_3)$, where $x^*_1>x^*_2>0>x^*_3$, then $k=2$, $l=1$, and $Y^*\prec\!\!\prec X^*$ means that $Y^*=(y^*_1,y^*_2,y^*_3)$, where 
\begin{gather*}
y^*_1\in\zeta_+ q^{\Z}, \quad x^*_1\ge y^*_1\ge x^*_2 t^{-1},\\
y^*_2\in \zeta_+ q^{\Z}t, \quad x^*_2\ge y^*_2>0,\\
y^*_3\in \zeta_-q^{\Z}, \quad 0>y^*_3\ge x^*_3.
\end{gather*}

\begin{theorem}\label{thm6.A}
Let $X^*\in\Om_n\subset\wt\Om_N$, where $n<N$.  Then the measure  $\wt\La^N_{N-1}(X^*, \ccdot)$ is concentrated on the set $\{Y^*\in\Om_n: Y^*\prec\!\!\prec X^*\}\subset\wt\Om_{N-1}$.
\end{theorem}

The remaining part of the section is devoted to the proof of this theorem. 

 \subsection{Preparation to proof} 
 
Let $X^*\in\Om_n\subset \wt\Om_N$ be fixed, $k:=k(X^*)$,  $l:=l(X^*)$ (see \eqref{eq6.B}), and $d:=N-n$. We assume $d>0$.  As usual, we enumerate the points of $X^*$ in the descending order: $X^*=(x^*_1>\dots>x^*_n)$. Suppose that both $k$ and $l$ are strictly positive (otherwise the argument is simplified, see section \ref{sect6.endproof}). Then  $x^*_k>0>x^*_{k+1}$. We will also use the alternative notation
$$
x^+:=x^*_k, \qquad x^-:=x^*_{k+1}.
$$

Let $A$ be a large positive integer. We insert between $0$ and $x^*_k$ the $d$-point configuration 
$$
X^0_A:=(x^+ q^A t, \, x^+ q^A t^2 ,\dots, x^+q^A t^d),
$$
and we set $X_A:=X^*\cup X^0_A$. Thus,
\begin{multline*}
X_A=(x^*_1,\dots,x^*_{k-1}, x^+;\;  x^+ q^A t, \, x^+q^A t^2,\dots, x^+ q^A t^d;\; x^-, x^*_{k+2},\dots,x^*_n)\\ =:(x_1,\dots, x_N).
\end{multline*}
Obviously, $X_A\in\Om_N$ and $X_A\to X^*$ as $A\to\infty$. Therefore, by the definition of the extended matrix $\wt\La^N_{N-1}$, the measure $\wt\La^N_{N-1}(X^*,\ccdot)$ is the weak limit of the measures $\LaN(X_A,\ccdot)$.  

Let $S_A$ denote the support of the pre-limit measure $\LaN(X_A,\ccdot)$: it consists of the configurations $Y\prec X_A$. Each $Y\in S_A$ contains the configuration
$$
Y^0_A:=(x^+ q^A t, \,x^+q^A  t^2,\dots, x^+q^A t^{d-1}),
$$
and we set $Y^*:=Y\setminus Y^0_A$.  

Thus, we may write
\begin{gather*}
Y=(y^*_1,\dots,y^*_k;\;  x^+ q^A t, \, x^+q^A t^2,\dots, x^+ q^A t^{d-1};\; y^*_{k+1},\dots,y^*_n),\\
Y^*=(y^*_1,\dots,y^*_k; y^*_{k+1},\dots,y^*_n).
\end{gather*} 
Note that the correspondence $Y\mapsto Y^*$ is one-to-one; we denote by $S^*_A$ the image of  the set $S_A$ under this correspondence.  

Next, observe that $y^*_{k+1}$ is the only point of the configuration  $Y$ that can be both to the left and to the right of zero. It is important for us to distinguish these two possibilities, so we write 
$$
S_A=S^-_A\sqcup S^+_A, \qquad S^-_A:=\{Y\in S_A: y^*_{k+1}<0\}, \quad S^+_A:=\{Y\in S_A: y^*_{k+1}>0\}
$$
and likewise
$$
S^*_A=S^{* -}_A\sqcup S^{* +}_A, \qquad S^{* -}_A:=\{Y^*\in S^*_A: y^*_{k+1}<0\}, \quad S^{* +}_A:=\{Y^*\in S^*_A: y^*_{k+1}>0\}.
$$
The bijection $S_A \leftrightarrow S^*_A$ gives rise to the bijections $S^\pm_A \leftrightarrow S^{* \pm}_A$.

Finally, observe that if $Y\in S^-_A$, then $Y^*\in \Om_n$ and $Y^*\pprec X^*$, while for $Y\in S^+_A$ this is wrong. 

\subsection{Reduction of the problem}

 The configurations $Y^*=(y^*_1,\dots,y^*_n)\in S^*_A$ satisfy the following constraints:

$\bullet$ Each of the $n-2$ points $y^*_1,\dots,y^*_{k-1}, y^*_{k+2},\dots,y^*_n$ may range only over a fixed finite set which does not depend on $A$. 

$\bullet$ The point $y^*_k$ may range over the set $\{x^+, x^+q,\dots,x^+q^A\}$, which is a finite geometric progression of the growing length $A+1$.

$\bullet$ The range of $y^*_{k+1}$ is the disjoint union of two infinite geometric progressions:
\begin{equation}\label{eq6.C1}
\{x^-, x^-q,x^-q^2,\dots\}\cup \{x^+t^d q^A,\, x^+t^d q^{A+1},\, x^+t^d q^{A+2},\dots\}.
\end{equation}

Let $\de>0$ be small and $S^-_A(\de)$ denote the subset of configurations $Y\in\ S^-_A$ satisfying at least one of the conditions $y^*_k\le\de$, $|y^*_{k+1}|\le\de$. 

Our first task is to reduce Theorem \ref{thm6.A} to the following two claims.

\smallskip

\emph{Claim} 1. We have 
\begin{equation*}
\lim_{\de\to 0}\sum_{Y\in S^-_A(\de)} \LaN(X_A, Y)=0 \quad \text{uniformly on $A$}.
\end{equation*}

\emph{Claim} 2. We have
\begin{equation*}
\lim_{A\to\infty}\sum_{Y\in S^+_A} \LaN(X_A, Y)=0.
\end{equation*}

\smallskip

\begin{proposition} These claims imply Theorem \ref{thm6.A}. 
\end{proposition}

\begin{proof}
Let us abbreviate
$$
M_A:=\LaN(X_A,\ccdot), \qquad M:=\wt\La^N_{N-1}(X^*,\ccdot).
$$
We can write $M_A=M^-_A+M^+_A$, where $M^\pm_A$ stands for the restriction of $M_A$ to the subset $S^\pm_A$ (more accurately, for this decomposition one should define $M^\pm_A$ as the result of multiplication of $M_A$ by the characteristic function of $S^\pm_A$). 

We know that $M$ is the weak limit of the measures $M_A$ as $A\to+\infty$. On the other hand, Claim 2 tells us that the total mass of $M^+_A$ tends to $0$ as $A\to+\infty$.  Therefore, $M$ is also the weak limit of the measures $M^-_A$.

Recall that a configuration $Y\in S^-_A$ differs from the corresponding configuration $Y^*\in S^{* -}_A$ only by the $(d-1)$-point configuration $Y^0_A$. The latter shrinks to $0$ as $A\to+\infty$. It follows that $Y$ and $Y^*$ get closer to each other as $A\to+\infty$, with respect to the uniform structure. Therefore, denoting by $M^{* -}_A$ the pushforward of $M^-_A$ under the bijection $S^-_A\leftrightarrow S^{*-}_A$, we conclude that the measures $M^-_A$ and $M^{* -}_A$ (which we regard as subprobability measures on the compact space $\wt\Om_N$) have a common weak limit. Thus, $M$ is the weak limit of the measures $M^{* -}_A$. 

We forget now about the compact space $\wt\Om_N$ and regard the measures $M^{* -}_A$ as subprobability measures on the countable discrete space 
$$
\Om^*_n:=\{Y^*\in\Om_n: Y^*\pprec X^*\}.
$$ 
We know that the total mass of $M^{* -}_A$ tends to $1$ as $A\to+\infty$. We also know  that $M^{* -}_A(Y^*)\to M(Y^*)$ for any fixed $Y^*\in \Om^*_n$ (this is a consequence of the weak convergence $M^{* -}_A\to M$). 

To finish the proof it remains to show that the family $\{M^{* -}_A\}$ is tight on $\Om^*_n$ (although our measures are not probability measures, only subprobability ones, this claim makes sense, because their masses tends to $1$). 

Observe that, on the discrete space $\Om^*_n$,  a configuration $Y^*\in\Om^*_n$ can escape to infinity only if $y^*_k\to0$ or $y^*_{k+1}\to0$, or both. From this it is seen that the desired tightness property is guaranteed by Claim 1. 
\end{proof}

We proceed to the proof of Claims 1 and 2. It is based on formula \eqref{eq5.R}, which we apply to $X=X_A=X^*\cup X^0_A$ and $Y=Y^*\cup Y^0_A$.  In our current notation, the configuration $X$ has $k+d$ points on the right of $0$, so that the parameter $k=k(X)$ in \eqref{eq5.L} should be replaced with $k+d$. Now \eqref{eq5.R} takes the form
\begin{multline}\label{eq6.E}
\LaN(X_A,Y)=\wt C_N(q,t)\frac{V(Y)}{V(X_A)}\cdot \prod_{r=1}^{N-1}|y_r|\cdot \prod_{r=1}^{N-1}\frac{(q^{m_r+1};q)_\infty}{(q^{m_r}t;q)_\infty} \\
\times \prod_{r=1}^{N-1} \prod_{\substack{s=1,\dots,N\\ s\ne r+\eps(r)}}\frac{(x_{r+\eps(r)}t/x_s;q)_{m_r}}{(x_{r+\eps(r)}q/x_s;q)_{m_r}}\cdot \prod_{\substack{j=1,\dots,N\\ j\ne k+d+1-\eps(k+d)}}\frac{(x_{k+d+1-\eps(k+d)}t/x_j;q)_\infty}{(x_{k+d+1-\eps(k+d))}q/x_j;q)_\infty},
\end{multline}
where 
$$
X_A=(x_1,\dots,x_N)=(x^*_1,\dots,x^*_{k-1}, x^+;\;  x^+ q^A t, \, x^+ q^A t^2,\dots, x^+ q^At^d;\; x^-, x^*_{k+2},\dots,x^*_n),
$$
$$
Y=(y_1,\dots,y_{N-1})=(y^*_1,\dots,y^*_k;\;  x^+ q^A t, \, x^+ q^A t^2,\dots, x^+ q^A t^{d-1};\; y^*_{k+1},\dots,y^*_n),
$$
$$
\eps(1)=\dots=\eps(k+d-1)=0, \quad \eps(k+d+1)=\dots=\eps(N-1)=1, 
$$
and
$$
\eps(k+d)=\begin{cases} 1, & Y\in S^-_A, \\ 0,  & Y\in S^+_A, \end{cases}
$$
so that
\begin{equation}\label{eq6.F}
k+d+1-\eps(k+d)=\begin{cases} k+d, & Y^*\in S^-_A, \\ k+d+1, & Y^*\in S^+_A. \end{cases}
\end{equation}
Next, the parameters $m_1,\dots,m_{N-1}$ are as follows:
$$
m_{k+1}=\dots=m_{k+d-1}=0;
$$
$m_1,\dots,m_{k-1}$ and $m_{k+d+1},\dots,m_{N-1}$ are bounded from above by certain constants that depend only on $X^*$ but not on $A$; finally, $0\le m_k\le A$ and $m_{k+d}\in\Z_{\ge0}$.

Further, we set $m':=m_{k}$, $m'':=m_{k+d}$. Then
$$
y_k=y^*_k=x^+q^{m'},  \qquad m'=0,\dots,A,
$$
and
$$
y_{k+d}=y^*_{k+1}=\begin{cases} x^-q^{m''}, & Y^*\in S^-_A,\\ x^+ t^d q^{m''+A}, & Y^*\in S^+_A, \end{cases}\quad 
\text{where} \quad m''\in\Z_{\ge0}.
$$

Let $f$ and $g$ be two expressions, possibly depending on $A$; then we write $f\lesssim g$ if $|f|\le \const |g|$ with some constant factor that does not depend on $A$. Let also agree that the symbol $\asymp$ will denote an equality up to a factor whose absolute value is bounded away from zero and infinity, uniformly on $A\to\infty$. 

We will establish the following bounds on the quantities $\LaN(X_A,Y)$: 

\begin{proposition}\label{prop6.C}
Assume $Y\in S^-_A$. Then
$$
\LaN(X_A,Y)\le\const\, t^{dm'} (\max(q,t))^{m''}
$$
with some constant factor which does not depend on $A$.
\end{proposition}

\begin{proposition}\label{prop6.D}
Assume $Y\in S^+_A$. Then 
$$
\LaN(X_A, Y)\asymp t^{dm'+dA} q^{m'+m''}.
$$
\end{proposition}

Let us show that Propositions \ref{prop6.C} and \ref{prop6.D} imply Claims 1 and 2, respectively. 

Indeed, in both claims, we have to estimate a double sum taken over two indices $(m',m'')$.

In Claim 1 we have to suppose that $-\de\le y_{k+d}=y^*_{k+1}<0$ with $\de\to0$, which amounts to saying that $m''\ge B$ with $B\to+\infty$. Summation over $m'$ produces an expression which is bounded by a constant, so we are left with the sum
$$
\sum_{m''\ge B} (\max(q,t))^{m''},
$$
which goes to $0$ as $B\to\infty$. 

In Claim 2 there is no similar constraint, but we have instead the factor $t^{dA}$; due to it the double sum goes to $0$ as $A\to+\infty$. 

Therefore, our problem is reduced to the proof of these two propositions. 

\subsection{Proof of Proposition \ref{prop6.C}}
We examine formula \eqref{eq6.E} taking into account the assumption $Y\in S^-_A$. 

According to \eqref{eq6.F} we have  $k+d+1-\eps(k+d)=k+d$. Thus, in the product over $j$, we have to substitute
$x_{k+d+1-\eps(k+d)}=x_{k+d}$. As  $A$ grows, the point $x_{k+d}$, which is in $X^0_A$, goes to $0$. From this it follows that the product over $j$ remains bounded. 

Next, examine the double product over $(r,s)$. Observe that if $m$ remains bounded, then any fraction of the form $(z;q)_m/(zt^{-1}q;q)_m$ also remains uniformly bounded, even if $z$ grows together with $A$. It follows that all fractions with $r$ distinct from $k$ and $k+d$ remain bounded. Hence we are left with 
\begin{equation}\label{eq6.G}
\prod_{r=k, k+d} \prod_{\substack{s=1,\dots,N\\ s\ne r+\eps(r)}}\frac{(x_{r+\eps(r)}t/x_s;q)_{m_r}}{(x_{r+\eps(r)}q/x_s;q)_{m_r}}
\end{equation}
Since $\eps(k)=0$ and $\eps(k+d)=1$, we have 
$$
k+\eps(k)=k, \qquad  k+d+\eps(k+d)=k+d+1,
$$
so that 
$$
x_{k+\eps(k)}=x^+, \qquad x_{k+d+\eps(k+d)}=x^-.
$$  
Therefore, \eqref{eq6.G} is equal to
$$
\prod_{x\in X\setminus\{x^+\}}\frac{(x^+t/x;q)_{m'}}{(x^+q/x;q)_{m'}}\cdot \prod_{x\in X\setminus\{x^-\}}\frac{(x^-t/x;q)_{m''}}{(x^-q/x;q)_{m''}}
$$
In this expression, all the fractions with $x\notin X^0_A$ remain bounded, hence the only relevant part is
\begin{multline}\label{eq6.H}
\prod_{x\in X^0_A}\frac{(x^+t/x;q)_{m'}}{(x^+q/x;q)_{m'}}\cdot \prod_{x\in X^0_A}\frac{(x^-t/x;q)_{m''}}{(x^-q/x;q)_{m''}}\\
=\prod_{i=1}^d \frac{(t^{1-i}q^{-A};q)_{m'}}{(t^{-i}q^{1-A};q)_{m'}}\cdot \prod_{i=1}^d\frac{((x^-/x^+)t^{1-i}q^{-A};q)_{m''}}{((x^-/x^+)t^{-i}q^{1-A};q)_{m''}}.
\end{multline}

To handle the resulting two products in \eqref{eq6.H} we need two lemmas.

\begin{lemma}\label{lemma6.A}
Let $u>\max(q,t)$ be fixed. For $m\le A$,
$$
\frac{(u q^{-A};q)_m}{(u t^{-1} q^{1-A};q)_m}\asymp \left(\frac tq\right)^m
$$
\end{lemma} 

\begin{proof}
We have 
\begin{gather*}
\frac{(u q^{-A};q)_m}{(u t^{-1} q^{1-A};q)_m}=\prod_{i=1}^m\frac{1-uq^{i-A-1}}{1-ut^{-1}q^{i-A}}=\left(\frac tq\right)^m \frac{(u^{-1}q^{A-m+1};q)_m}{(u^{-1}tq^{A-m};q)_m}\asymp \left(\frac tq\right)^m,
\end{gather*}
where the last step is justified by the fact that, due to the assumption on $u$,
$$
0<u^{-1}q^{A-m+1}\le u^{-1}q<1 \quad \text{and} \quad 0<u^{-1}tq^{A-m}\le u^{-1}t<1.
$$
\end{proof}

\begin{lemma}\label{lemma6.B}
Let $w>0$ be fixed. Then
$$
\frac{(-wq^{-A};q)_m}{(-wt^{-1}q^{1-A};q)_m}\asymp \left(\frac tq\right)^{\min(m,A)}.
$$
\end{lemma}

\begin{proof}
Suppose $m\le A$. We have 
\begin{gather*}
\frac{(-wq^{-A};q)_m}{(-wt^{-1}q^{1-A};q)_m}=\prod_{i=1}^m\frac{1+w q^{-A+i-1}}{1+wt^{-1}q^{-A+i}}=\left(\frac tq\right)^m\prod_{i=1}^m\frac{1+w^{-1} q^{A-i+1}}{1+w^{-1}q^{A-i}}\\
=\left(\frac tq\right)^m\frac{(-w^{-1}q^{A-m+1};q)_m}{(-w^{-1}tq^{A-m};q)_m}\asymp \left(\frac tq\right)^m,
\end{gather*}
where last step is justified by the fact that the quantities $w^{-1}q^{A-m+1}$ and $w^{-1}tq^{A-m}$ are bounded from above by a constant which does not depend on $A$ (here we use the hypothesis $m\le A$). 

Suppose now $m>A$. Then we have
$$
\frac{(-wq^{-A};q)_m}{(-wt^{-1}q^{1-A};q)_m}
=\frac{(-wq^{-A};q)_A}{(-wt^{-1}q^{1-A};q)_A}
\cdot \frac{(-wq;q)_{m-A}}{(-wt^{-1}q^1;q)_{m-A}}\asymp \left(\frac tq\right)^A.
$$

This completes the proof of the lemma.
\end{proof}

\begin{corollary}\label{cor6.A}
The expression \eqref{eq6.H} is \quad $\asymp \left(\dfrac tq\right)^{dm'+d\min(m'',A)}. $
\end{corollary}

\begin{proof}
The first and the second product on the right-hand side of \eqref{eq6.H} are estimated by applying Lemma \ref{lemma6.A} and Lemma \ref{lemma6.B}, respectively. In Lemma \ref{lemma6.A} we set $u=t^{1-i}$,  and in Lemma \ref{lemma6.B} we set $-w=(x^-/x^+)t^{1-i}$. In both cases, $i=1,\dots,d$. 
\end{proof}

\begin{lemma}\label{lemma6.C}
Under the assumption that $Y\in S^-_A$, we have 
$$
\frac{V(Y)}{V(X_A)}\prod_{r=1}^{N-1}|y_r|\lesssim q^{dm'+m''+(d-1)\min(m'',A)}.
$$
\end{lemma}

\begin{proof}
Indeed,
$$
\prod_{i=1}^{N-1}|y_r|\asymp q^{m'+m''+A(d-1)},
$$
$$
V(X_A)\asymp V(X^0_A)\asymp q^{Ad(d-1)/2},
$$
\begin{multline*}
V(Y)\asymp V(Y^0_A)\cdot(y_k-y_{k+d})\cdot \prod_{y\in\Y^0_A}(y-y_{k+d})\cdot\prod_{y\in Y^0_A}(y_k-y)\\
\asymp q^{A(d-1)(d-2)/2}\cdot(x^+q^{m'}+|x^-|q^{m''})\cdot \prod_{i=1}^{d-1}(x^+t^i q^A+|x^-|q^{m''})\cdot q^{m'(d-1)}\\ \lesssim q^{A(d-1)(d-2)/2}\cdot q^{(d-1)\min(m'',A)}\cdot q^{m'(d-1)}.
\end{multline*}
This implies the desired bound.
\end{proof}

From Corollary \ref{cor6.A} and Lemma \ref{lemma6.C} we obtain
$$
\LaN(X_A,Y)\le\const\, t^{dm'+d\min(m'',A)} q^{m''-\min(m'',A)}\le \const\, t^{dm'+\min(m'',A)} q^{m''-\min(m'',A)} ,
$$
because $d\ge1$.

It remains to check the inequality
$$
t^{\min(m'',A)} q^{m''-\min(m'',A)}\le (\max(q,t))^{m''}.
$$

If  $m''\le A$, then it turns into 
$$
t^{m''}\le (\max(q,t))^{m''},
$$
which is obvious. 

Finally, suppose  $m''>A$. Then the inequality takes the form 
$$
t^A q^{m''-A}\le (\max(q,t))^{m''}.
$$

If $t\le q$, then this means $t^A q^{m''-A}\le q^{m''}$ or else $t^A q^{-A}\le1$, which holds true (because $t\le q$). 

If $t>q$, then this means $t^A q^{m''-A}\le t^{m''}$ or else $q^{m''-A}\le t^{m''-A}$, which also holds true (because $m''-A>0$ and $t>q$). 

This completes the proof of Proposition \ref{prop6.C}.

\subsection{Proof of Proposition \ref{prop6.D}}

Now we reexamine formula \eqref{eq6.E} assuming $Y\in S^+_A$. 

According to \eqref{eq6.F}, we have  $k+d+1-\eps(k+d)=k+d+1$. Because $x_{k+d+1}=x^-$, the product over $j$ in \eqref{eq6.E} takes the form
\begin{multline}\label{eq6.I}
\prod_{j\ne k+d+1}\frac{(x^-t/x_j;q)_\infty}{(x^-q/x_j;q)_\infty}\asymp \prod_{x\in X^0_A}\frac{(x^-t/x;q)_\infty}{(x^-q/x;q)_\infty}\\
=\prod_{i=1}^d \frac{((x^-/x^+)t^{1-i}q^{-A};q)_\infty}{((x^-/x^+)t^{-i}q^{1-A};q)_\infty}\asymp\left(\frac tq\right)^{dA},
\end{multline}
where the last step is justified by the following lemma, which we apply for $w:=-(x^-/x^+)t^{1-i}$,

\begin{lemma}\label{lemma6.D}
Let $w>0$ be fixed. As $A\to+\infty$, 
$$
\frac{(-wq^{-A};q)_\infty}{(-wt^{-1}q^{1-A};q)_\infty}\asymp \left(\frac tq\right)^A.
$$
\end{lemma}

\begin{proof}
We have
$$
\frac{(-wq^{-A};q)_\infty}{(-wt^{-1}q^{1-A};q)_\infty}=\frac{(-wq^{-A};q)_A}{(-wt^{-1}q^{1-A};q)_A} \, \frac{(-w;q)_\infty}{(-wt^{-1}q;q)_\infty}=\const\,\frac{(-wq^{-A};q)_A}{(-wt^{-1}q^{1-A};q)_A} \asymp \left(\frac tq\right)^A,
$$
where the last step follows from Lemma \ref{lemma6.B} in which we take $m=A$. 
\end{proof}

Now we turn to the double product over $(r,s)$ in \eqref{eq6.E}. Again, its relevant part is \eqref{eq6.G}.  We still have $\eps(k)=0$, but now $\eps(k+d)$ equals $0$ (not $1$, as above). The result is that
$$
k+\eps(k)=k, \qquad k+d+\eps(k+d)=k+d,
$$
and hence
$$
x_{k+\eps(k)}=x_k=x^+, \qquad x_{k+d+\eps(k+d)}=x_{k+d}\in X^0_A.
$$ 
It follows that the part of \eqref{eq6.G} related to $r=k+d$ is uniformly bounded, hence we are left with 
\begin{equation}\label{eq6.J}
 \prod_{\substack{s=1,\dots,N\\ s\ne k}}\frac{(x_k t/x_s;q)_{m_k}}{(x_k q/x_s;q)_{m_k}}\asymp  \prod_{x\in X^0_A}\frac{(x^+ t/x;q)_{m'}}{(x^+q/x;q)_{m'}}=\prod_{i=1}^d\frac{(t^{1-i}q^{-A};q)_{m'}}{(t^{-i}q^{1-A};q)_{m'}}\asymp \left(\frac tq\right)^{dm'},
\end{equation}
where the last step is justified with the help of Lemma \ref{lemma6.A}.

Combining \eqref{eq6.I} and \eqref{eq6.J} we obtain:

\begin{corollary}\label{cor6.B}
The expression in the second line of \eqref{eq6.E}  is \; $
\asymp \left(\dfrac tq\right)^{dm' +dA}.
$
\end{corollary}

The next lemma is an analogue of Lemma \ref{lemma6.C}:

\begin{lemma}\label{lemma6.E}
Under the assumption that $Y\in S^+_A$, we have 
$$
\frac{V(Y)}{V(X_A)}\prod_{r=1}^{N-1}|y_r|\asymp q^{(d+1)m'+m''+dA} .
$$
\end{lemma}

\begin{proof}
Indeed,
$$
\prod_{i=1}^{N-1}|y_r|\asymp q^{m'+(m''+A)+A(d-1)},
$$
$$
V(X_A)\asymp V(X^0_A)\asymp q^{Ad(d-1)/2},
$$
\begin{multline*}
V(Y)\asymp V(Y^0_A)(y_k-y_{k+d})\cdot \prod_{y\in\Y^0_A}(y-y_{k+d})\cdot\prod_{y\in Y^0_A}(y_k-y)\\
\asymp q^{A(d-1)(d-2)/2}\cdot q^{m'}\cdot q^{A(d-1)}\cdot q^{m'(d-1)}
\end{multline*}
This implies the desired bound.
\end{proof}

Corollary \ref{cor6.B} and Lemma \ref{lemma6.E} together imply Proposition \ref{prop6.D}.

\subsection{End of proof of Theorem \ref{thm6.A}}\label{sect6.endproof}
We have finished the proof of the theorem in the case when $0<k<n$. It remains to examine the two extreme cases $k=n$ and $k=0$. By symmetry, they are equivalent, so we examine only the case $k=n$. Then all points of $X_A$ and $Y$ are on the right of $0$. We have $N=k+d$,
$$
X_A=(x_1,\dots,x_N)=(x^*_1,\dots,x^*_k;\;  x^*_k q^A t, \, x^*_k q^A t^2,\dots, x^*_k q^At^d),
$$
$$
Y=(y_1,\dots,y_{N-1})=(y^*_1,\dots,y^*_k;\;  x^*_k q^A t, \, x^*_k q^A t^2,\dots, x^*_k q^A t^{d-1}),
$$
and
$$
\eps(1)=\dots=\eps(k+d-1)=0.
$$

In the configuration $Y$, each of the points except $y_k=y^*_k$ is either fixed or ranges over a fixed finite set. As for $y_k$, it may take the values of the form $x^*_kt q^m$, where $0\le m\le A-1$. 

In the present situation the theorem reduces to the following claim: 
$$
\lim_{B\to+\infty}\sum_{Y:\, m\ge B}\LaN(X_A,Y)=0 \quad \text{uniformly on $A$}.
$$

But this follows from the next proposition, which is a simplified version of Propositions \ref{prop6.C} and \ref{prop6.D}:

\begin{proposition}\label{prop6.E}
In the case $k=n$ we have
$$
\LaN(X_A,Y)\asymp t^{dm}.
$$
\end{proposition}

\begin{proof}
We turn again to formula \eqref{eq6.E}. The product over $j$ now disappears and the only relevant part has the form
\begin{equation*}
\frac{V(Y)}{V(X_A)}\cdot \prod_{r=1}^{N-1}|y_r|\cdot \prod_{r=1}^{N-1} \prod_{\substack{s=1,\dots,N\\ s\ne r}}\frac{(x_rt/x_s;q)_{m_r}}{(x_r q/x_s;q)_{m_r}},
\end{equation*}

Arguing as in Lemma \ref{lemma6.E} we obtain
$$
\frac{V(Y)}{V(X_A)}\cdot \prod_{r=1}^{N-1}|y_r|\asymp q^{dm}.
$$

Next, the product over $(r,s)$ is handled with the aid of Lemma \ref{lemma6.A}, and the result is \; $\asymp\left(\dfrac tq\right)^{dm}$. 

These two estimates yield the desired result.
\end{proof}

This completes the proof of Theorem \ref{thm6.A}.

\section{Boundaries of projective chains:  general facts}\label{sect7}

\subsection{The boundary}\label{sect7.1}
Recall (section \ref{sect1.1}) that a \emph{projective chain} $\{S_N, L^N_{N-1}\}$ consists of an infinite sequence $S_1,S_2,\dots$ of countable sets linked by stochastic matrices  $L^N_{N-1}$ of format $S_N\times S_{N-1}$, where $N=2,3,\dots$\,. 

Recall that the symbol $\P(\ccdot)$ denotes the set of probability measures on a given measurable space. The matrix  $L^N_{N-1}$ determines a map $\P(S_N)\to\P(S_{N-1})$, which we write as  $M\mapsto M L^N_{N-1}$ (here any measure $M\in\mathscr P(S_N)$ is interpreted as a row-vector whose coordinates are indexed by $S_N$). 

The maps $\P(S_N)\to\P(S_{N-1})$ allow us to form the projective limit space $\varprojlim \P(S_N)$. By the very definition, an element of this space is an infinite sequence $\{M_N\in \P(S_N): N=1,2,\dots\}$ with the property that $M_N L^N_{N-1}=M_{N-1}$ for all $N\ge2$; such sequences are called \emph{coherent systems}. 

In what follows we assume that the space $\varprojlim\P(S_N)$ is nonempty. It possesses a natural structure of a convex set, which  gives sense to the following definition.

\begin{definition}\label{def7.A}
By the \emph{boundary} of a projective  chain $\{S_N, L^N_{N-1}: N=2,3,\dots\}$ we mean the set $S_\infty:=\operatorname{Ex}(\varprojlim\P(S_N))$ of extreme points of $\varprojlim\P(S_N)$.
\end{definition}

Given $X\in S_\infty$, we denote by $M^{(X)}=\{M^{(X)}_N\}$ the coherent system represented by $X$. 
In the next theorem we use the natural Borel structure on $\varprojlim\P(S_N)$ generated by the cylinder sets.    

\begin{theorem}\label{thm7.A}
The set\/ $S_\infty$ is a Borel subset of the space\/ $\varprojlim\P(S_N)$, so that we may form the space $\P(S_\infty)$ of probability Borel measures on\/ $S_\infty$.

For every coherent system $M=\{M_N\}$ there exists a unique measure $\sigma\in\P(S_\infty)$ such that $M=\int_{S_\infty} M^{(X)}\sigma(dX)$ in the sense that
$$
M_N(Y)=\int_{S_\infty} M^{(X)}_N(Y)\sigma(dX) \qquad \text{\rm for every $N=1,2,\dots$ and every $Y\in\Om_N$.}
$$

Conversely, every measure $\sigma\in\P(S_\infty)$ generates in this way a coherent system, so that we obtain a bijection  $\P(S_\infty)\leftrightarrow\varprojlim\P(S_N)$.
\end{theorem}

\begin{proof} See Olshanski \cite[Theorem 9.2]{Ols-2016}.
\end{proof}

For later use it is convenient to slightly reformulate this result. Let us introduce the alternative notation
$$
L^\infty_N(X,Y):=M^{(X)}_N(Y), \qquad Y\in S_N.
$$
We may regard $L^\infty_N$ as a Markov kernel: this simply means that $L^\infty_N(X,\ccdot)$ is a probability measure on $S_N$ for any fixed $X\in S_\infty$, and the function $X\mapsto L^\infty_N(X,Y)$ is a Borel measurable function on $S_\infty$ for any fixed $Y\in S_N$. Next, we rename $\si$ by $M_\infty$. In this notation, Theorem \ref{thm7.A} claims that there is a one-to-one correspondence $\{M_N\}\leftrightarrow M_\infty$ between coherent systems and probability measures on $S_\infty$ given by
$$
M_N=M_\infty L^\infty_N, \qquad N=1,2,3,\dots,
$$
or, in more detail,
\begin{equation}\label{eq7.C}
M_N(Y)=\int_{S_\infty} M_\infty(dX) L^\infty_N(X,Y), \qquad N=1,2,3,\dots, \quad Y\in S_N.
\end{equation}

We call $M_\infty$ the \emph{boundary measure} of a given coherent system $\{M_N\}$. It is tempting to say that $M_\infty$ is the limit of the measures $M_N$ as $N\to\infty$.  One cannot do it in the abstract setting, because $M_\infty$ and the $M_N$'s live on distinct spaces. However, in a number of concrete models one can use their specific properties and deduce from Theorem \ref{thm7.A} that the $M_N$'s do converge to $M_\infty$ in some natural sense. In particular, this can be done in our case, see Theorem \ref{thm8.C} below.

\subsection{The path space}

Given two elements $X\in S_N$ and $Y\in S_{N-1}$, we write $X\triangleright Y$ or equivalently $Y\triangleleft X$ if the matrix entry $L^N_{N-1}(X,Y)$ is nonzero (hence strictly positive). 

A \emph{finite path} of length $N$ is a sequence $(X(1),\dots,X(N)$, where
$X(i)\in S_i$ for all $i=1,\dots,N$ and  $X(1)\triangleleft X(2)\triangleleft\dots\triangleleft X(N)$.  The set of all such paths will be denoted by $\Pi_N$.

Likewise, an \emph{infinite path} is an infinite sequence $(X(1)\triangleleft X(2)\triangleleft\dots)$, where $X(i)\in S_i$ for all $i=1,\dots$\,. The set of all such paths is denoted by $\Pi$ and called the \emph{path space}. 

For each $N\ge2$ there is a natural projection $\Pi_N\to\Pi_{N-1}$:
$$
(X(1),\dots,X(N-1),X(N))\;\mapsto\; (X(1),\dots,X(N-1)).
$$
Evidently, $\Pi$ is the projective limit of the sets $\Pi_N$ with respect to these projections. We equip $\Pi$ with the corresponding Borel structure. 

An \emph{elementary cylinder set of depth $N$} in $\Pi$ is the set of all infinite paths with a prescribed beginning $(X(1),\dots,X(N))$; let us denote such a set by $C(X(1),\dots,X(N))$. 

\begin{definition}
We say that a probability Borel measure $\mathscr M$ on $\Pi$ is a \emph{Gibbs measure} if the following condition holds. Let $N$ and $X\in S_N$ be arbitrary, and consider all elementary cylinder sets $C(X(1),\dots,X(N))$ with $X(N)=X$. Then we require that
\begin{multline}\label{eq7.A}
\mathscr M(C(X(1),\dots,X(N))=m(X)L^N_{N-1}(X(N),\,X(N-1))\\ 
\times L^{N-1}_{N-2}(X(N-1),\,X(N-2)) \dots L^2_1(X(2),\,X(1)),
\end{multline}
where $m(X)\ge0$ is a quantity that depends on $X$ only. (Cf. \cite[sect. 7.4]{BO-2017}.)
\end{definition}

\begin{proposition}\label{prop7.A}
There is a natural bijective correspondence $\mathscr M\leftrightarrow \{M_N\}$ between Gibbs measures and coherent systems.
\end{proposition}

\begin{proof}
Let $\mathscr M$ be a Gibbs measure. For each $N$,  we define a measure $M_N\in\P(S_N)$ by setting $M_N(X)=m(X)$ for $X\in S_N$, where $m(X)$ is taken from \eqref{eq7.A}. It is immediately checked the $M_N$ is a probability measure and the sequence $M_1,M_2,\dots$ is a coherent system.

Conversely, let $\{M_N\}$ be a coherent system. For each $N$, we define a probability measure $\mathscr M_N$ on $\Pi_N$ by setting  
\begin{multline*}
\mathscr M_N(C(X(1),\dots,X(N))=M_N(X(N))L^N_{N-1}(X(N),\, X(N-1))\\
\times  L^{N-1}_{N-2}(X(N-1), \,X(N-2))\dots L^2_1(X(2),\,X(1)),
\end{multline*}
where $(X(1),\dots,X(N))\in\Pi_N$. The measures obtained in this way are consistent with the projections $\Pi_N\to\Pi_{N-1}$. Hence, by Bochner's theorem (see Bochner \cite[Theorem 5.1.1]{Bochner} or Parthasarathy \cite[Ch. V]{Parth}), they give rise to a probability measure $\mathscr M$ on $\Pi$. By the very construction, it is a Gibbs measure. 
\end{proof}

Let $\pi=(X(N))$ and $\pi'=(X'(N))$ be two infinite paths; let us say that they are \emph{equivalent} (and then write $\pi\sim\pi'$) if  they have the same \emph{tail}, that is,  $X(N)=X'(N)$ for all $N$ large enough. Let $G$ be the group of all bijections $g:\Pi\to \Pi$ such that $g\pi\sim\pi$ for every path $\pi$ and $g\pi\ne\pi$ for finitely many paths $\pi$ only. This is a countable group of transformations of $\Pi$. Associated with the action of $G$ on $\Pi$ is a $1$-cocycle $c(g,\pi)$: if $\pi=(X(N))$ and $g\pi=(X'(N))$, then
$$
c(g,\pi):=\prod_{N=2}^\infty\frac{L^N_{N-1}(X'(N),X'(N-1))}{L^N_{N-1}(X(N),X(N-1))}.
$$
The product on the right is actually finite and hence is well defined. 

The notion of Gibbs measures can be reformulated as follows: these are precisely those probability measures $\mathscr M\in\P(\Pi)$ that are $G$-quasiinvariant and consistent with the cocycle $c(g,\pi)$, that is, for a test function $f$ on $\Pi$,
\begin{equation}\label{eq7.B}
\int_{\Pi}f(g^{-1}\pi)\mathscr M(d\pi)=\int_{\Pi}f(\pi)c(g,\pi)\mathscr M(d\pi) \qquad \forall g\in G.
\end{equation}

This fact is used in the next proposition. Before to state it, observe that the notion of Gibbs measures given above can be extended, in a natural way, to finite measures (not necessarily probability ones). Next, if $\mathscr M$ is a finite measure on $\Pi$ and $A\subset \Pi$ is a Borel subset, then we denote by $\mathscr M\big|_A$ the restriction of $\mathscr M$ to $A$, which we regard again as a measure on $\Pi$. 

\begin{proposition}\label{prop7.B}
Let $A\subset\Pi$ be a Borel subset, which is saturated with respect to the tail equivalence relation {\rm(}that is, $A$ consists of whole equivalence classes{\rm)}. If $\mathscr M$ is a Gibbs measure, then so is $\mathscr M\big|_A$.
\end{proposition}

\begin{proof} 
We have $\mathscr M\big|_A=\chi_A\mathscr M$, where $\chi_A$ denotes the characteristic function of $A$. Our assumption on $A$ means that $\chi_A$ is $G$-invariant. It follows that if $\mathscr M$ satisfies \eqref{eq7.B}, then so is $\chi_A\mathscr M$. This concludes the proof.
\end{proof}

\subsection{Decomposition on singular and nonsingular components}\label{sect7.3}
 
Now we apply the general formalism described above to two concrete projective chains, $\{\Om_N, \LaN\}$ and $\{\wt\Om_N, \wt\La^N_{N-1}\}$; the results are used below in the proof of Theorem \ref{thm8.B}.

Let $\Pi$ and $\wt\Pi$ denote the spaces of infinite paths for $\{\Om_N, \LaN\}$ and $\{\wt\Om_N, \wt\La^N_{N-1}\}$, respectively. Evidently, $\Pi$ is a subset of $\wt\Pi$.  Let us say that a path $\pi\in\wt\Pi$ is \emph{nonsingular} if it is contained in $\Pi\subset\wt\Pi$; otherwise it is called \emph{singular}. 

Write a path $\pi\in\wt\Pi$ as a sequence $\{X(N)\in\wt\Om_N: N=1,2,\dots\}$. In this notation, $\pi$ is nonsingular if and only if $X(N)\in\Om_N$ for each $N$; that is, $X(N)$ must contain exactly $N$ points. Therefore, $\pi$ is singular if this condition fails, that is, there exists and index $K$ such that $X(K)$ contains less than $K$ points, say, $k<K$ points. Then, as is seen from Theorem \ref{thm6.A}, for all $N>K$ the number of points in $X(N)$ is also equal to $k$. 

Thus, the nonsingular paths have the form
$$
X(1)\prec X(2)\prec \dots, \qquad X(N)\in\Om_N, \quad N=1,2,\dots,
$$
while the singular paths have the form
$$
X(1)\prec\dots \prec X(k)\prec\!\!\prec X(k+1)\prec\!\!\prec X(k+2)\prec\!\!\prec\dots, \qquad 
$$
where $X(N)\in\Om_k\subset\wt\Om_N$ for all $N\ge k$, with a certain $k$. Recall that the meaning of symbol $\prec\!\!\prec$ is explained in Definition \ref{def6.B}. 

From this description we obtain a \emph{stratification} of the space $\wt\Pi$: 
$$
\wt\Pi=\Pi\sqcup \bigsqcup_{k=0}^\infty \Pi_k,
$$
where $\Pi_k$ is formed by those paths $\{X(N): N=1,2,\dots\}$ for which $X(N)\in\Om_k$ for all $N\ge k$. 

\begin{lemma}\label{lemma7.A}
Each of the strata $\Pi, \Pi_0,\Pi_1,\dots$ is a saturated Borel subset.
\end{lemma}

\begin{proof}
The fact that the strata are saturated follows directly from their definition. Next, for each $k=0,1,2,\dots$, the stratum $\Pi_k$ consists precisely of the paths $\pi=(X(N))$ with the property $X(k+1)\in\Om_k$.  It follows that $\Pi_k$ is the union of countably many elementary cylinder sets and hence a Borel set. This in turn implies that $\Pi$ is a Borel set, too. 
\end{proof}

\begin{corollary}\label{cor7.C}
Any extreme Gibbs measure on the space $\wt\Pi$ is concentrated on one of its strata $\Pi,\Pi_0,\Pi_1,\dots$\,. 
\end{corollary}

\begin{proof}
Let $\mathscr M$ be an arbitrary finite Gibbs measure on the path space $\wt\Pi$. Lemma \ref{lemma7.A} makes it possible to restrict $\mathscr M$ to any of the strata. Moreover, the resulting measure (denote it by $\mathscr M_\infty$ or $\mathscr M_k$) will be a Gibbs measure by virtue of Proposition \ref{prop7.B}.  We obtain the decomposition 
$$
\mathscr M=\mathscr M_\infty+\mathscr M_0+\mathscr M_1+\dots
$$
in which each all the components are Gibbs measures. In the case when $\mathscr M$ is an extreme probability Gibbs measure it must coincide with one of its components, all other components being equal to zero. This completes the proof. 
\end{proof}

\section{Proof of Theorems B and C}\label{sect8}

\subsection{The boundary of the chain $\{\wt\Om_N,\wt\La^N_{N-1}\}$} 

Let $\Sym$ denote the algebra of symmetric functions. Observe that for any configuration $X\in\wt\Om$ and any positive integer $k$, the sum $\sum_{x\in X}|x|^k$ is finite. It follows that for any $f\in\Sym$ and any $X\in\wt\Om$, the value of $f$ at $X$ makes sense: namely, we enumerate the points $x\in X$ in an arbitrary way and set
$$
f(X):=f(x_1,x_2,\dots), \quad \text{where} \quad X=(x_1,x_2,\dots)
$$
(we add infinitely many $0$'s if $X$ if finite). An important remark is that $f$ is a continuous function in the topology of the space $\wt\Om$. 
 
We denote by $P_\nu=P_\nu(x_1,x_2,\dots;q,t)$ the \emph{Macdonald symmetric function} with index $\nu\in\Y$ and parameters $q$ and $t$ (Macdonald \cite[Ch. VI, sect. 4]{M}). Its value at $X\in\wt\Om$ is denoted by $P_\nu(X;q,t)$. 

Given $X\in\wt\Om$ (see Definition \ref{def6.A} (ii)), we denote by $[X]$ the smallest closed interval of $\R$ containing all points of $X$.

\begin{theorem}\label{thm8.A}
The elements of the boundary of the chain\/ $\{\wt\Om_N,\wt\La^N_{N-1}\}$ can be para\-met\-rized by the configurations $X\in\wt\Om$. 

More precisely, to every $X\in\wt\Om$ there corresponds a coherent system $M^{(X)}=\{M^{(X)}_K: K=1,2,\dots\}$; here  the $K$th measure $M^{(X)}_K\in\P(\wt\Om_K)$ is concentrated on the compact set $\{Y\in\wt\Om_K: Y\subset [X]\}$ and is uniquely determined by the relations
\begin{equation}\label{eq8.B}
\sum_{Y\in\wt\Om_K}M^{(X)}_K(Y)\frac{P_{\nu\mid K}(Y;q,t)}{(t^K; q,t)_\nu}=P_\nu(X;q,t),
\end{equation}
where $\nu$ is an arbitrary partition with $\ell(\nu)\le K$. The coherent families $M^{(X)}$ are pairwise distinct and are precisely the extreme ones. 

Furthermore, the Borel structure on the boundary coincides with the Borel structure of the space $\wt\Om$ determined by its topology. 
\end{theorem}

In the particular case $t=q$ this result was proved in \cite[Theorem 6.2]{Ols-2016}, and the same argument works in the general case. So we only sketch the proof and refer to \cite{Ols-2016} for more details.

\begin{proof}[Sketch of proof]

\emph{Step} 1. For $N>K$ we set 
$$
\wt\La^N_K:=\wt\La^N_{N-1}\wt\La^{N-1}_{N-2}\dots\wt\La^{K+1}_K;
$$
this is a stochastic matrix of format $\wt\Om_N\times\wt\Om_K$.  Below we use the following direct generalization of \eqref{eq6.A}: if $\nu\in\Y(K)$, then
\begin{equation}\label{eq8.D}
\sum_{Y\in\wt\Om_K}\wt{\La}^N_K(X,Y)\frac{P_{\nu\mid K}(Y;q,t)}{(t^K;q,t)_\nu}=\frac{P_{\nu\mid N}(X;q,t)}{(t^N;q,t)_\nu},\qquad \forall X\in\wt\Om_N.
\end{equation}
Note that $\wt\La^N_K(X,Y)$ vanishes unless $Y\subset[X]$. 

Let $X\in\wt\Om$ be arbitrary. Take a sequence $\{X(N)\in\wt\Om_N: N=1,2,\dots\}$ such that $X(N)\subset[X]$ and $X(N)\to X$ (such a sequence always exists). We claim that for any fixed $K=1,2,\dots$ there exists a weak limit 
\begin{equation}\label{eq8.A}
M^{(X)}_K:=\lim_{N\to\infty}\wt\La^N_K(X(N),\ccdot)\in\P(\wt\Om_K).
\end{equation}

Indeed, substitute $X=X(N)$ into \eqref{eq8.D} and rewrite the resulting equality in the form
\begin{equation}\label{eq8.D1}
\left\langle \wt{\La}^N_K(X(N),\ccdot), \; \frac{P_{\nu\mid K}(\ccdot;q,t)}{(t^K;q,t)_\nu}\right\rangle=\frac{P_{\nu\mid N}(X(N);q,t)}{(t^N;q,t)_\nu},
\end{equation}
where the angular brackets denote the canonical pairing between measures and functions.

Fix $\nu$ and let $N\to\infty$. Then $(t^N;q,t)_\nu\to1$ and the right-hand side of \eqref{eq8.D1} tends to $P_\nu(X;q,t)$. Thus, the left-hand side also has a limit for each $\nu$. Since the measures $\wt{\La}^N_K(X(N),\ccdot)$ are concentrated on a compact set, they have a weak limit, as it is seen from the argument of Lemma \ref{lemma5.uni}. 

By the very construction, the limit measure, denoted by  $M^{(X)}_K$, is concentrated on the compact set $\{Y\in\wt\Om_K: Y\subset [X]\}$ and is uniquely determined by the relations \eqref{eq8.B}. In particular, it does not depend on the choice of the approximation $X(N)\to X$. Furthermore, the sequence $\{M^{(X)}_K\}$ is a coherent system, and different configurations $X\in\wt\Om$ lead to different coherent systems. All these claims are proved exactly as in \cite{Ols-2016}. 

(Note that a phrase in \cite{Ols-2016} has to be corrected: there, in the proof of Theorem 6.2,  the beginning of step 1, it is written that any sequence $\{X(N)\}$ converging to $X$ is `regular', meaning that  the measures $\wt\La^N_K(X(N),\ccdot)$ converge in a stronger sense, which is not true in general. However, we do not need this; for our purpose it suffices that these measures converge weakly.) 

\emph{Step} 2.  Let $\{M_K:K=1,2,\dots\}$ be an extreme coherent system. By a general theorem (see \cite[Theorem 6.1]{OO-IMRN}), there exists a sequence $\{X(N)\in\wt\Om_N\}$ such that, as $N$ goes to infinity, $\wt\La^N_K(X(N),Y)\to M_K(Y)$ for every $K$ and every $Y\in\wt\Om_K$. A fortiori, for every $K$, the measures $\wt\La^N_K(X(N), \ccdot)$ converge to $M_K$ weakly. In particular, this holds for $K=1$ which in turn implies that the measures $\wt\La^N_1(X(N), \ccdot)$ form a tight family of probability measures on $\wt\Om_1$. Now we apply Proposition \ref{prop8.A} (see below); it tells us that there exists a positive number $a$ such that $X(N)\subset [-a,a]$ for each $N$. Because the subset $\{X\in\wt\Om: X\subset [-a,a]\}\subset\wt\Om$ is compact, the sequence $\{X(N)\}$ has a limit point in $\wt\Om$. Therefore, one may choose a subsequence of indices $N$ such that, along this subsequence, $X(N)$ converges to some element $X\in\wt\Om$. Applying the result of step 1 we see that $M_K=M^{(X)}_K$ for every $K$. We conclude that the extreme coherent systems are contained among the systems of the form $\{M^{(X)}_K\}$.  

\emph{Step} 3. Here we prove the converse claim: any coherent system of the form $\{M^{(X)}_K\}$ is extreme. The argument is the same as in \cite{Ols-2016}, with Schur symmetric functions being replaced by  Macdonald symmetric functions. 

\emph{Step} 4. Here we apply a general fact about Borel maps to prove the final claim of the theorem. This claim is also necessary to justify an argument in step 3. 
\end{proof}

\subsection{A condition of tightness}\label{sect8.2}
Our task here is to prove the following proposition, which was used in the argument above, on step 2.

\begin{proposition}\label{prop8.A}
Let\/ $\{X(N)\in\wt\Om_N: N=1,2,\dots\}$ be a sequence of configurations such that the corresponding sequence $\{\wt\La^N_1(X(N),\ccdot)\}$ of probability measures on\/ $\wt\Om_1$ is tight. Then there exists $a>0$ such that $X(N)\subset[-a,a]$ for all $N$. 
\end{proposition}

First we state a lemma. 

\begin{lemma}\label{lemma8.A}
Let $X$ be a nonempty configuration from $\wt\Om_N$, where $N\ge2$, and let $x_0$ denote the point of $X$ with maximal absolute value, so that $x_0$ is either the leftmost or the rightmost point {\rm(}in the case these endpoints of $X$ have the same absolute value we take as $x_0$ any of them{\rm)}. 

The number $\wt\La^N_1(X,x_0)$ is bounded from below by a universal positive constant{\rm:}
\begin{equation}\label{eq8.B1}
\wt\La^N_1(X,x_0)\ge c:=\frac{\prod\limits_{m=1}^\infty(t^m;q)_\infty}{(-1;q)_\infty\prod\limits_{m=1}^\infty(-t^m;q)_\infty}>0.
\end{equation}
\end{lemma}

Observe that the proposition immediately follows from the lemma. Indeed, if the configurations $X(N)$ are not uniformly bounded, then one can choose a subsequence of numbers $N_1<N_2<\dots$ such that for the corresponding configurations $X(N_i)$, at least one endpoint goes to infinity. By virtue of the lemma, this means that the measure $\wt\La^N_1(X_i,\ccdot)$ has an atom of size $\ge c>0$ that escapes to infinity as $i\to\infty$. But this contradicts the tightness assumption. Thus, it remains to prove the lemma. 

\begin{proof}[Proof of the lemma]
First of all note that the two infinite products on the right-hand side of \eqref{eq8.B1} converge. Indeed, to see this, write each of them as a double product
$$
\prod_{m=1}^\infty (\pm t^m;q)_\infty=\prod_{m=1}^\infty\prod_{n=0}^\infty(1\mp t^mq^n)
$$
and observe that 
$$
\sum_{m=1}^\infty\sum_{n=0}^\infty t^mq^n<\infty.
$$

We follow the proof of Lemma 4.3 in \cite{Ols-2016} which in turn relies on computations in \S3 of that paper. 

\emph{Step} 1. Recall that $\wt\Om_1=\Om_0\cup\Om_1$, where $\Om_0$ consists of the empty configuration and $\Om_1=\L$. Since $X$ is assumed to be nonempty,  the measure $\wt\La^N_1(X,\ccdot)$ is concentrated on $\Om_1$: here we use Theorem \ref{thm6.A}. Thus, we may regard $\wt\La^N_1(X,\ccdot)$ as a measure on $\L$.

\emph{Step} 2. Let us show that
\begin{equation}\label{eq8.C}
\sum_{y\in\L}\wt\La^N_1(X,y)\frac{(yz^{-1}t^N;q)_\infty}{(yz^{-1};q)_\infty}=\prod_{x\in X}\frac{(xz^{-1}t;q)_\infty}{(xz^{-1};q)_\infty}, \quad z\in\C\setminus\R,
\end{equation}
cf. \cite[Proposition 3.1]{Ols-2016}. 

Indeed, setting $K=1$ and $\nu=(n)$ in \eqref{eq8.D} we get
\begin{equation}\label{eq8.E}
\sum_{y\in\L}\wt{\La}^N_1(X,y)\frac{(t^N;q)_n}{(t;q)_n}y^n=P_{(n)\mid N}(X;q,t),\qquad \forall X\in\Om_N, \quad n\in\Z_{\ge0}.
\end{equation}
Assume first that $|z|$ is large, multiply the both sides of \eqref{eq8.E} by $Q_{(n)\mid1}(z^{-1};q,t)$ (the univariate Macdonald $Q$-polynomial) and sum over all $n\in\Z_{\ge0}$. On the right-hand side we obtain the right-hand side of \eqref{eq8.C}, by virtue of the fundamental Cauchy identity for the Macdonald symmetric functions \cite[ch. VI, (4.13)]{M}. 

Let us turn to the left-hand side. Here we may interchange the order of summation, which gives us 
$$
\sum_{y\in\L}\wt\La^N_1(X,y)\left\{\sum_{n=0}^\infty\frac{(t^N;q)_n}{(t;q)_n}Q_{(n)\mid1}(z^{-1};q,t)y^n\right\}.
$$
We can compute the interior sum. From the definition of the Macdonald $Q$-functions (see \cite[ch. VI, (4,12), (4.11), and 6.19)]{M}) it follows that 
$$
Q_{(n)\mid1}(z^{-1};q,t)=\frac{(t;q)_n}{(q,q)_n}z^{-n},
$$
Therefore, the interior sum is
$$
\sum_{n=0}^\infty\frac{(t^N;q)_n}{(q;q)_n}(yz^{-1})^n=\frac{(yz^{-1}t^N;q)_\infty}{(yz^{-1};q)_\infty},
$$
where the last equality follows from the $q$-binomial formula \cite{GR}. This gives the desired equality \eqref{eq8.C}. Finally, we get rid of the assumption that $|z|$ is large by using analytic continuation. 

\emph{Step} 3. Let us derive from \eqref{eq8.C} the equality
\begin{equation}\label{eq8.F}
\wt\La^N_1(X,x_0)=\frac{(t;q)_\infty}{(t^N;q)_\infty}\,\frac{\prod\limits_{x\in X\setminus\{x_0\}}(xx_0^{-1}t;q)_\infty}{\prod\limits_{x\in X\setminus\{x_0\}}(xx_0^{-1};q)_\infty}.
\end{equation}

Indeed, the right-hand side of \eqref{eq8.C} is a meromorphic function in $z\in\C\setminus\{0\}$. It has a pole at $z=x_0$ with the residue 
\begin{equation}\label{eq8.C1}
\Res_{z=x_0}\left\{\prod_{x\in X}\frac{(xz^{-1}t;q)_\infty}{(xz^{-1};q)_\infty}\right\}=x_0\frac{(t;q)_\infty}{(q;q)_\infty}\,\frac{\prod\limits_{x\in X\setminus\{x_0\}}(xx_0^{-1}t;q)_\infty}{\prod\limits_{x\in X\setminus\{x_0\}}(xx_0^{-1};q)_\infty}.
\end{equation}

On the hand, let us compute the same residue by looking at the left-hand side of \eqref{eq8.C}. Recall that the support of the measure $\wt\La^N_1(X,\ccdot)$ is contained in $[X]$, the smallest closed interval containing $X$. From this and by the very definition of $x_0$ we conclude that only the summand with $y=x_0$ contributes. Therefore, the residue in question is equal to 
\begin{equation}\label{eq8.C2}
\wt\La^N_1(X,x_0)\Res_{z=x_0}\left\{\frac{(x_0z^{-1}t^N;q)_\infty}{(x_0z^{-1};q)_\infty}\right\} =\wt\La^N_1(X,x_0) x_0\frac{(t^N;q)_\infty}{(q;q)_\infty}.
\end{equation} 

Equating \eqref{eq8.C1} to \eqref{eq8.C2} we obtain \eqref{eq8.F}. 

\emph{Step} 4. It remains to find a lower bound for the right-hand side of \eqref{eq8.F}. By the definition of $x_0$, the numerator can be estimated as follows 
$$
(t;q)_\infty \prod\limits_{x\in X\setminus\{x_0\}}(xx_0^{-1}t;q)_\infty\ge (t;q)_\infty (t^2;q)_\infty\dots (t^N;q)_\infty\ge\prod_{m=1}^\infty(t^m;q)_\infty.
$$
Likewise, the denominator can be estimated as follows
$$
(t^N;q)_\infty \prod\limits_{x\in X\setminus\{x_0\}}(xx_0^{-1};q)_\infty\le (-1;q)_\infty(-t;q)_\infty \dots (-t^{N-2};q)_\infty\le(-1;q)_\infty\prod_{m=1}^\infty(-t^m;q)_\infty.
$$
This completes the proof of \eqref{eq8.B}.
\end{proof}

\subsection{Proof  of Theorem B}

The next theorem is similar to that of Theorem \ref{thm8.A}. Recall that the space $\Om_\infty$ was introduced in Definition \ref{def6.A}.

\begin{theorem}\label{thm8.B}
The elements of the boundary of the chain\/ $\{\Om_N,\La^N_{N-1}\}$ can be parametrized by the configurations $X\in\Om_\infty$. 

More precisely, to every $X\in\Om_\infty$ there corresponds a coherent system $M^{(X)}=\{M^{(X)}_K: K=1,2,\dots\}$; here  the $K$th measure $M^{(X)}_K\in\P(\Om_K)$ is concentrated on the compact set $\{Y\in\Om_K: Y\subset [X]\}$ and is uniquely determined by the relations
\begin{equation}\label{eq8.B3}
\sum_{Y\in\Om_K}M^{(X)}_K(Y)\frac{P_{\nu\mid K}(Y;q,t)}{(t^K; q)_\nu}=P_\nu(X;q,t),
\end{equation}
where $\nu$ is an arbitrary partition with $\ell(\nu)\le K$. The coherent families $M^{(X)}$ are pairwise distinct and are precisely the extreme ones. 

Furthermore, the Borel structure on the boundary coincides with the Borel structure of the ambient space $\wt\Om$ determined by its topology. 
\end{theorem}

This result is a reformulation (with a slight refinement) of Theorem B (see section \ref{results2}).

\begin{proof}
The results of Section \ref{sect7} show that the boundary of the chain $\{\Om_N,\La^N_{N-1}\}$ is contained in the boundary of the chain $\{\wt\Om_N,\wt\La^N_{N-1}\}$. We know (Theorem \ref{thm8.A}) that the latter boundary is the space $\wt\Om$, and we are going to prove that the former boundary is its subset $\Om_\infty$. After that the remaining claims will follow from the corresponding claims of Theorem \ref{thm8.A}.

Proposition \ref{prop7.A} allows us to switch to the language of Gibbs measures. Let $X\in\wt\Om$, $\{M^{(X)}_K:K=1,2,\dots\}$ be the corresponding coherent system, and $\mathscr M^{(X)}$ denote the corresponding Gibbs measure on the path space $\wt\Pi$. We know that $\mathscr M^{(X)}$ is extreme. Therefore, by virtue of Corollary \ref{cor7.C},  $\mathscr M^{(X)}$ is concentrated on one of the strata of $\wt\Pi$, and  the boundary under question is the set of those configurations $X\in\wt\Om$ for which the corresponding stratum  is $\Pi$, not $\Pi_k$. 

Thus, it suffices to show that $\mathscr M^{(X)}$ is concentrated on some $\Pi_k$ if and only if  $X\in\wt\Om\setminus\Om_\infty$. We proceed to the proof of this claim.  

Suppose that there exists $k$ such that $\mathscr M^{(X)}$ is concentrated on the stratum $\Pi_k$. This implies that for any $K>k$, the measure $M^{(X)}_K$ is concentrated on $\Om_k\subset \wt\Om_K$.  Now let $\nu$ be an arbitrary partition with $\ell(\nu)>k$. Take an arbitrary $K\ge\ell(\nu)$ and observe that the polynomial $P_{\nu\mid K}(\ccdot;q,t)$ vanishes on the subset $\Om_k\subset\wt\Om_K$. Then \eqref{eq8.B} shows that $P_\nu(X;q,t)=0$. In particular, all elementary symmetric functions $e_n$ with $n>k$ vanish at $X$ (here we use the fact that $e_n$ coincides with $P_{(1^n)}(\ccdot;q,t)$). Therefore, the generating function
$$
1+\sum_{n=1}^\infty e_n(X)z^n=\prod_{x\in X}(1+x z)
$$
is a polynomial in $z$ of degree at most $k$. 

On the other hand, this function vanishes at each point of the form $z=-x^{-1}$ with $x\in X$. This implies that $X$ has at most $k$ points, so that $X\in\wt\Om\setminus\Om_\infty$.

Conversely, suppose that $X\in\wt\Om\setminus\Om_\infty$, so that $X\in\Om_k\subset\wt\Om$ for some $k$. For each $N>k$ let $X(N)$ denote the same configuration $X$ regarded as an element of $\Om_k\subset \wt\Om_N$. Then for each $K$ the limit relation \eqref{eq8.A} holds. It shows that the measure $M^{(X)}_K$ is concentrated on the subset of configurations with at most $k$ points. This in turn implies that $\mathscr M^{(X)}$ cannot be concentrated on $\Pi$. .
\end{proof}

\subsection{Proof of Theorem C}
The next theorem contains Theorem C (section \ref{results3}). 

\begin{theorem}\label{thm8.C}
Let $\{M_N\}$ be a coherent system of probability distributions for the chain $\{\Om_N, \LaN\}$ or $\{\wt\Om_N, \wt\La^N_{N-1}\}$ and let $M_\infty$ be the corresponding boundary measure on $\Om_\infty$ or\/ $\wt\Om$, respectively. Then $M_N\to M_\infty$ in the weak topology of the space $\P(\wt\Om)$.
\end{theorem}

\begin{proof}
Recall that the space $\wt\Om$ is locally compact and suppose first that $M_\infty$ is compactly supported. Then there exists an interval $[a,b]\subset\R$ such that $M_\infty$ is concentrated on the compact subset $\wt\Om[a,b]:=\{X\in\wt\Om: X\subset[a,b]\}$. It follows that the same also holds for all measures $M_N$. The symmetric functions form a dense subset of the Banach space $C(\wt\Om[a,b])$, hence it suffices to prove that, as $N\to\infty$,
\begin{equation}\label{eq8.B2}
\langle M_N, P_\nu(\ccdot;q,t)\rangle \to \langle M_\infty, P_\nu(\ccdot;q,t)\rangle
\end{equation}
for any partition $\nu$, 
where, as before,  the angular brackets denote the canonical pairing between measures and functions, and $M_N$ is regarded as a measure on $\wt\Om$. 

On the other hand, for large enough $N$ we may write
$$
\langle M_N, P_\nu(\ccdot;q,t)\rangle=\langle M_N, P_{\nu\mid N}(\ccdot;q,t)\rangle
$$
and then it follows from \eqref{eq8.B} and \eqref{eq8.B3} that, as $N\to\infty$, 
$$
\langle M_N, P_{\nu\mid N}(\ccdot;q,t)\rangle=(t^N;q,t)_\nu\langle M_\infty, P_{\nu\mid N}(\ccdot;q,t)\rangle.
$$
Since $(t^N;q,t)_\nu\to1$, this implies \eqref{eq8.B2}.

In the general case we may write $M_\infty$ as a convex combination of two probability measures,
$$
M_\infty=(1-\eps) M'_\infty+\eps M''_\infty,
$$
where $M'_\infty$ is compactly supported and  $\eps>0$ is a small parameter. Denote by $\{M'_N\}$ and $\{M''_\infty\}$ the coherent systems corresponding to $M'_\infty$ and $M''_\infty$, respectively. For an arbitrary fixed bounded continuos function $F$ on  $\wt\Om$ we have
\begin{multline*}
|\langle M_\infty,F\rangle - \langle M_N,F\rangle|\le 
(1-\eps)|\langle M'_\infty,F\rangle - \langle M'_N,F\rangle|  + \eps|\langle M''_\infty,F\rangle - \langle M''_N,F\rangle|\\
\le (1-\eps)|\langle M'_\infty,F\rangle - \langle M'_N,F\rangle|  +2\eps \Vert F\Vert.
\end{multline*}

By virtue of the above argument, as $N$ gets large, $|\langle M'_\infty,F\rangle - \langle M'_N,F\rangle|$ goes to $0$. It follows that 
$$
\lim_{N\to\infty}|\langle M_\infty,F\rangle - \langle M_N,F\rangle|=0,
$$
which completes the proof. 
\end{proof}

\bigskip

\noindent \textbf{Funding}

\medskip

\noindent This work was supported by the Russian Science Foundation [project 20-41-09009].

\bigskip

\noindent\textbf{Acknowledgments}
\medskip

\noindent I am grateful to Cesar Cuenca and an anonymous referee for valuable comments.

\end{document}